\documentclass[11pt]{article}
\usepackage{amsfonts}
\usepackage{alltt}
\usepackage{amssymb,verbatim,ifthen}
\usepackage{dsfont}
\usepackage[all]{xy}
\usepackage{color}
\usepackage{graphicx}
\usepackage{amsmath}
\usepackage{amsthm}
\usepackage{mathabx}
\usepackage{pdfsync}
\usepackage{listliketab}
\usepackage{natbib}
\usepackage{thm-restate}
\usepackage{subcaption} 
\usepackage{hyperref}
\usepackage{textcomp}
\usepackage{mathtools}
\usepackage{theoremref}

\usepackage{multirow,array}
\usepackage{pgfplots}
\pgfplotsset{compat=1.7}
\usetikzlibrary{intersections}
\usetikzlibrary{calc}
\usetikzlibrary{patterns}

\usepackage{enumitem}

\usepackage{comment}

\usepackage[textwidth=30mm]{todonotes}

\newcommand{\YT}[1]{\todo[color=red!50,author=\textbf{Yevgeny},inline]{\small #1\\}}

\newcommand{\XV}[1]{\todo[color=blue!100!white!30,author=\textbf{Xavier},inline]{\small #1\\}}

\renewcommand\thmcontinues[1]{Continued}

\def\split{true}
\ifthenelse{\equal{\split}{true}} {
\textwidth = 440pt \oddsidemargin = 2pt \evensidemargin = 2pt
}
{}
\newcommand{\ignore}[1]{}

\renewenvironment{proof}[1][Proof]{\textbf{#1} }{\ \rule{0.5em}{0.5em}}
\newtheorem{theorem}{Theorem}
\newtheorem*{theorem*}{Theorem}

\newtheorem{definition}{Definition}
\newtheorem{obs}{Observation}
\newtheorem{proposition}{Proposition}

\newtheorem{example}{Example}
\newtheorem{corollary}{Corollary}

\renewcommand\thmcontinues[1]{Continued}

\newcommand\ubar[1]{\underline{#1}} 
\def\1{{1\hskip-2.5pt{\rm l}}} 
\def\s{\sigma} 
\def\D{\Delta} 
\def\G{\Gamma} 
\def\RR{\mathbb{R}} 
\def\NN{\mathbb{N}} 
\def\E{\mathbb{E}} 
\def\Tau{\mathcal{T}} 
\def\f{f} 

\DeclareMathOperator*{\argmax}{arg\,max}
\DeclareMathOperator*{\argmin}{arg\,min}

\makeatother

\begin{document}

\title{Repeated Games with Switching Costs: \\Stationary vs History-Independent Strategies
\thanks{The authors wish to thank E. Solan, D. Lagziel, J. Flesch, and G. Ashkenazi-Golan  for their highly valuable comments. 
YT acknowledges the support of the French National Research Agency Grant ANR-17-EURE-0020, and by the Excellence Initiative of Aix-Marseille University -- A*MIDEX.
XV acknowledges the financial support by the National Agency for Research, Project CIGNE (ANR-15-CE38-0007-01).
}}
\author{Yevgeny Tsodikovich\thanks{Aix Marseille Univ, CNRS, AMSE, Marseille, France. e-mail: \textsf{yevgets@gmail.com}},
Xavier Venel\thanks{Dipartimento di Economia e Finanza, LUISS, Viale Romania 32, 00197 Rome, Italy,  
\textsf{xvenel@luiss.it}, https://orcid.org/0000-0003-1150-9139}, 
Anna Zseleva\thanks{
	Maastricht University, School of Business and Economics, Dept. of Quantitative Economics, P.O. Box 616, 6200 MD Maastricht, The Netherlands, e-mail: \textsf{anna.zseleva@maastrichtuniversity.nl}}}

\maketitle

\thispagestyle{empty}

\lineskip=2pt\baselineskip=5pt\lineskiplimit=0pt

\begin{abstract}
\noindent We study zero-sum repeated games where the one player has to pay a certain cost each time he changes his action.
Our contribution is twofold.
First, we show that the value of the game exists in stationary strategies, depending solely on the previous action of the minimizing player, not the entire history. 
We provide important properties of the value and the optimal strategies as a function of the ratio between the switching costs and the stage payoffs.
The strategies exhibit a robustness property and typically do not change with a small perturbation of the switching costs.
Second, we consider a case where the minimizing player is limited to playing simpler strategies that are history-independent. 
Here too, we characterize the (minimax) value and the strategies for obtaining it.
Moreover, we present several bounds on the loss due to this limitation.
\end{abstract}

\bigskip
\noindent {\emph{Journal} classification numbers: C72, C73.}

\bigskip



\newpage
\lineskip=1.8pt\baselineskip=18pt\lineskiplimit=0pt \count0=1

\section{Introduction}
An environmental inspector is in charge of a large region which comprises of several potential illegal dump sites.
His job is to patrol between the different sites to prevent pollution and penalize violators.
Traveling, however, infers costs to him since while in motion all the sites are left unguarded.
In addition, some physical movement costs might apply, such as the cost of fuel.
Therefore, while planning his route, he needs to take into account both the effectiveness of each possible route in minimizing the pollution and the costs associated with traveling (``switching costs''). 

This model, namely the \emph{traveling inspector model (TIM)}, was presented in \cite{filar1986traveling} (see also \cite{filar1985player} and \cite{filar1983interactive}) and was used to study different situations which include inspection of remote locations, such as environmental protection \citep{jorgensen2010dynamic}, arms control verification \citep{o1994game}, and scheduling problems of sensors \citep{yavuz2007analysis}.
Recently, a similar model was proposed in the field of cybersecurity \citep{rass2014numerical,wachter2018security}, where the locations are virtual rather than physical but there is still a problem of monitoring all the locations at once and the switches are costly.
For example, when changing the configuration of a computer system costs could include downtime of servers and hourly rates of staff.
In this paper we consider the TIM as well and study how changing the ratio between the switching costs and the pollution costs affect the value of the inspector.
In addition, we compare this value to the payoff that can be obtained when only using simpler history-independent strategies, as is done in some applications \citep{GDO2020paper}.

The TIM is an example of switching costs (or ``menu costs'' \citep{akerlof1985can,akerlof1985near}) which appear naturally in different economic settings, such as markets \citep{beggs1992multi} and bargaining \citep{caruana2007multilateral}.
They can reflect the real cost of changing one's action (costs of setting up machinery to increase production, shut-down costs of reducing production and adapting machinery from one production line to another) or be a method of modeling bounded rationality when changing actions might be ``hard'' \citep{lipman2000switching,lipman2009switching}. 
In a management quality context, switching costs can be related to changing a company policy or marketing strategy.
In these examples, the underlying game is a non-zero-sum game and the addition of switching costs creates new opportunities for the players to cooperate to increase their payoffs.
This occurs when the switching costs either eliminate existing equilibria (without switching costs) with low payoffs or turn a Pareto-optimal strategy profile into an equilibrium, such as the cooperative strategy profile in the repeated Prisoner's Dilemma (even the finitely repeated one, see \cite{lipman2000switching,lipman2009switching}).
On the other hand, switching costs might impede punishments by making switching to a punishment strategy costly.

To the best of our knowledge, apart from studies focused on specific games (such as the Blotto game in \cite{shy2019colonel}), switching costs have never been thoroughly addressed theoretically in the setting of repeated normal-form zero-sum games.
\cite{schoenmakers2008repeated} studied an equivalent zero-sum model, in which a constant bonus is granted each time the previous action is not changed. 
In this paper, we generalize their work\footnote{Receiving a bonus for repeating an action is equivalent to paying a fine for not repeating the action. Thus, up to an additive constant, their model is equivalent to a private case of our model studied in Section \ref{s_constantsc}.} to non-constant switching costs, discuss how optimal strategies can be approximated using the simpler history-independent strategies, and study limit cases of small and large switching costs (compared to the payoffs of the game) and action-independent switching costs.

\noindent\textbf{Our model.}
We consider a repeated normal-form zero-sum game.
At each time step, the minimizing player pays the maximizing player both the ``standard'' outcome of the game and an additional fine if he switched his previous action (the maximizing player incurs no switching costs).
This happens in the TIM, since during the switches all sites are not monitored, which benefits the polluter in the same way it harms the inspector.
This assumption also corresponds to a worst-case analysis a player might perform when assuming the others act as an adversary.
For example, when the game is non-zero-sum, this results in the reservation payoff that each player can guarantee for himself, and is the basis to defining feasible and individually rational payoffs.
We use this to analyze non-zero-sum games with switching costs as part of another project (work in progress).

Unlike the models most closely resembling ours, \cite{schoenmakers2008repeated} and \cite{lipman2009switching}, we allow the switching costs to depend on the action.
In the TIM, this captures the idea that there may be different costs involved in reaching different locations, based on distance and topography.
We assume that there is a multiplicative factor $c$ applied to the switching costs.
When we increase this factor, the switching costs increase, while the ratio between the costs of two different switches remains the same (an increase in fuel costs increases the cost of movement between locations without changing the distances between them). 
By changing factor $c$, therefore, we can study how the value of the game and the optimal strategies change when the switching costs change uniformly, without affecting the internal structure of the costs.


The additivity of the payoffs naturally becomes a factor when dealing with multi-objective minimization.
\cite{rass2014numerical} proved that any equilibrium in a game with multiple objectives is a minimax strategy in a game with a single objective, where the payoff function is some weighted average of all the objective functions.
Hence, if the minimizing player wishes to minimize both his ``regular'' payoff in the repeated game and his switching costs, an alternative approach is to consider a game where the two are weighted and summed, as in our model.

Finally, as in the TIM, we use undiscounted payoffs.
This makes sense as pollution ``does not bear interest'', and pollution today is just as bad as pollution tomorrow.
Alternatively, in cybersecurity, the number of stages played over a short time (hours) is very large and they all equally important, so it is possible to assume that the discount factor is very close to patience while the horizon is infinite.

\noindent\textbf{Optimal strategies.}
A game with switching costs is equivalent to a stochastic game \citep{Shapley1953stochastic} where the states correspond to the previous action taken by the minimizing player \cite{filar1986traveling}.
Since this game is a single-controller stochastic game (only he controls the states), there exists an optimal strategy that is \emph{stationary} (depends solely on the state and not the time) and can be computed using one of several standard tools (see for example the technique in Appendix \ref{app_acoe}, as well as \cite{filar1980algorithms,stern1975stochastic,filar1986traveling,raghavan2002computing}, \cite{raghavan2003finite} and others).

In practice, however, such an approach suffers from two disadvantages.
First, if in the original game the player has $n$ pure actions, a stationary strategy in the stochastic game assigns a mixed action to each state, which yields $n^2-n$ decision variables.
This poses a computational challenge that leads to the introduction of new simpler strategies, at the expense of some payoff.
Second, in some applications such as cybersecurity, there is an additional requirement that the mixed actions of the minimizing player (defender) should be the same for every state \citep{rass2017defending,wachter2018security}, so only history-independent strategies are permitted.

We therefore also consider strategies that satisfy this additional constraint, and do not depend on the past.
We call strategies that utilize the same mixed action at each time step ``\emph{static strategies}''.
Although the problem of finding an optimal static strategy\footnote{Typically the game has no value in static  strategies. Hence, hereafter, when discussing the value in such strategies, we are referring to the minimax value, and when discussing optimal strategies, the minimax strategies.} is NP-hard, there exist algorithms that efficiently find an approximate solution \citep{GDO2020paper}.
In this paper, we compare the optimal payoff in static  strategies to the value of the game (achievable through stationary strategies) in different scenarios and provide bounds on the loss incurred from applying this constraint on strategies. 
Moreover, we provide theoretical background for the use of static strategies and explain several phenomena reported by the empirical works of \cite{GDO2020paper} and \cite{rass2017defending}, such as the fact that optimal strategies only change slightly (if at all) in response to a small change in the ratio between switching costs and stage payoffs.

\noindent\textbf{Our contribution.} 
The contribution of this paper is twofold.
First, we study the problems arising from zero-sum games with switching costs and provide a complete characterization of the payoff function and the optimal stationary strategies.
We show that these strategies belong to a particular finite set that depends only on the underlying one-shot game, not the switching costs.
This can be used as a basis for an algorithm to compute them.
Second, we study the value that can be obtained by using simpler static strategies and provide bounds on the loss incurred by using them instead of stationary ones.
We show that in some cases, the value is obtained by static strategies, so there is no loss at all.
In special cases (such as switching costs that are independent of the actions, as in \cite{lipman2000switching,lipman2009switching,schoenmakers2008repeated}), we can also provide a complete characterization of the payoff function in static strategies and show that in this case too, the optimal strategies belong to a particular finite set which depends solely on the one-shot game.


In particular, we find that the value function in stationary strategies is piece-wise linear in the weight of the switching costs, $c$.
This means that the optimal strategy depends only on the segment in which $c$ is situated, not on the exact value of $c$.
Moreover, our finding has two important implications.
First, it implies that the player can identify the optimal strategy without knowing the exact weight of the switching costs relative to the game (which is true in the setting of \cite{rass2014numerical}). 
Second, it allows the player to assess the improvement in value that can be obtained if the switching costs are reduced, as the optimal strategy does not change for small changes of $c$.
For example, if the traveling inspector can increase his movement speed by a small amount (say, a $1\%$ improvement), our result ensures him that he does not need to adapt his strategy to the new speed (it is still optimal) and that the improvement in payoff is linear in this $1\%$

We show that the properties of the value function under stationary strategies (mainly the piece-wise linearity) also hold for the value under static strategies, assuming that switching costs are independent of the actions being switched.
Otherwise, the optimal strategies might strongly depend on $c$.
We provide additional results on the optimal strategies in specific games by adding stronger assumptions on the structure of the game or the switching costs.
For example, we show that when switching costs are symmetric (the amount paid when switching from pure action $i$ to $j$ is the same as when switching from $j$ to $i$), the bound on the difference between the value in stationary and static strategies is smaller than when switching costs are not symmetric.
We then generalize our work to wider classes of games and obtain general bounds on the loss arising from using history-independent strategies relative to history-dependent ones.
In particular, we show that at least a quarter (after normalization) of the optimal payoff can be obtained using very simple history-independent strategies.

\noindent\textbf{Structure of the paper.} 
This section is followed by a short review of the related literature.
In Section \ref{s_model} we formally present the model.
In Section \ref{s_results} we provide our results.
We start with a general characterization of the value functions (Section \ref{s_propv}), followed by a series of theorems concerning the bound on the difference between the value in stationary and in static strategies.
Then, in Section \ref{s_gen}, we consider static strategies for stochastic games and deduce appropriate bounds.
A short discussion follows in Section \ref{s_conclusions}.
In Appendix \ref{app_acoe} we present the average cost of optimality (ACOE) algorithm which is used to compute optimal stationary strategies and the value of the game.
The solved example in the Appendix also provides some intuition regarding our results and in particular \thref{lem:proportiesofvc}.

\subsection{Related Literature}

Switching costs have been studied in different economic settings, mainly in non-zero-sum games, where the main questions concerned the existence of a Folk Theorem, players' ability to credibly punish each other, and the commitment that switching costs add.
A common simplifying assumption is that all switching costs are identical and independent of the actions being switched \citep{lipman2000switching,lipman2009switching}.
This requires only one additional parameter to cover switching costs and allows concise analysis of the effect of the magnitude of switching costs on different aspects of the game.
We allow switching costs to depend on actions, but take the same approach, adding one variable and studying its effect on the outcome of the game.
In our model, this variable is a multiplicative factor affecting all switching costs in the same manner.
This allows us to capture more real-life situations and consider wider-ranging possibilities in different models.

In addition to this generalization, our paper fills two gaps in the switching costs literature.
First, except for \cite{schoenmakers2008repeated}, zero-sum games are generally overlooked in explorations of switching costs.
Although they can be viewed as merely a sub-case of general games, the considerations and questions addressed for zero-sum games are substantially different from those  for non-zero-sum games and therefore require a specific focus.
This paper, by considering zero-sum games, both addresses zero-sum situations and provides a reference point for non-zero-sum games.
For example, when determining the worst-case payoff (the individually rational payoff), the auxiliary game to consider is a zero-sum game.

It should be noted that even \cite{schoenmakers2008repeated} did not consider games with switching costs, but rather games with bonuses, where one player receives a bonus from repeating the same action, due to a ``learning by doing'' effect.
They argued that if the opponent is a computer, then it does not ``learn'' and should not receive a bonus for playing the same action.
However, using equivalent transformations, the model of \cite{schoenmakers2008repeated} can be rewritten to fit our model, with the additional restriction that switching costs are independent of actions.
Thus, our paper generalizes their work in terms of optimal value in stationary and static strategies (which they call ``simple strategies'') and their dependence on switching costs relative to the payoffs of the game.
Since their model is more specific than ours, they are able to provide additional results.
For example, by assuming that the one-shot payoff matrix is regular, they provide a formula for the value of the game. 

Second, although static strategies are considered by many (\cite{schoenmakers2008repeated} and \cite{bomze2020does} are two prominent examples), the main focus is on the computational aspects of the optimal static strategies or approximations of them \citep{GDO2020paper}.
We provide several theoretical results regarding such strategies and the value obtained by them.
Our work, therefore, contributes both by examining the quality of approximations of optimal strategies and by suggesting some ways to improve the computational aspects of the algorithms.
In particular, we provide a bound on the difference between the optimal stationary value and the optimal static value.
Using our theoretical bound, the ``price of using static strategies'' can be calculated and compared to other factors, such as computational complexity.
Such a comparison has never been made before, since solving the game and finding the value in stationary strategies is generally very challenging computationally.

\section{The Switching Costs Model}\label{s_model}
A \emph{zero-sum game with switching costs} is a tuple $\G=(A,S,c)$ where $A=(a_{ij})$ is an $m\times n$ matrix, $S=(s_{ij})$ is an $n\times n$ matrix and $c\geq 0$.
At each time step $t$, Player~1 (the row player, the maximizer, she) chooses an integer\footnote{We assume action sets are finite. Without loss of generality, we identify them as integers representing the corresponding row or column in the payoff matrix.} $i(t)$ in the set $\{1,\ldots,m\}$ and Player~2 (the column player, the minimizer, he) chooses an integer $j(t)$ in the set $\{1,\ldots,n\}$.
The stage payoff that Player~2 pays Player~1 is $a_{i(t)j(t)}+c s_{j(t-1)j(t)}$, so Player~2 is penalized for switching from previous action $j(t-1)$ to action $j(t)$ by $c s_{j(t-1)j(t)}$ .
We assume that $s_{kj}\geq 0$ and $s_{jj}=0$, for all $k,j$. Moreover, we normalize $S$ by assuming $\min\limits_{s_{kj}\neq0}s_{kj}=1$.
Due to the multiplicative factor $c$, this normalization is without loss of generality.
In some of the examples below we avoid it for ease of exposition.

At each time period, the players are also allowed to play mixed actions.
A mixed action is a probability distribution over the player's action set.
Note that at time $t$, $j(t-1)$ is already known, even if Player~2 played a mixed action at time $t-1$.
The process repeats indefinitely, and the payoff is the undiscounted average ($\liminf$)\footnote{Other methods of defining the undiscounted payoff can be considered, all of which coincide in this model according to \cite{filar1980algorithms}.}.
More precisely, let $(\sigma,\tau)$ be a pair of strategies in the repeated game and $(\sigma(t),\tau(t))$ the mixed actions played at stage $t$ (given the history).
We define
\begin{equation}\label{eq:gamma_liminf}
\gamma(\sigma,\tau)=\E_{\sigma,\tau}\left(\liminf\limits_{T\to\infty} \tfrac{1}{T} \sum\limits_{t=1}^{T} u(\sigma(t),\tau(t)) \right)
\end{equation}
where $u(i(t),j(t))=a_{i(t)j(t)}+c s_{j(t-1)j(t)}$ is the stage payoff defined above (no switching cost is paid at $t=1$).

When $c=0$, this simply involves indefinitely repeating the one-shot matrix game $A$.
Without loss of generality, we assume in the theorems that the minimal and the maximal entries in $A$ are $0$ and $1$ (but not in the examples, for clarity).
We define $\mathcal{A}$ as the set of optimal strategies for Player~2 in the one-shot game $A$, $v$ the value of the game and $\bar{v}$ the minmax value in pure strategies.  
In the rest of the paper, we will assume that $A$ and $S$ are fixed and examine how the optimal strategies of the players and their payoffs change with $c$, the relative weight of the switching costs compared to the stage payoff of game $A$.

This repeated game is equivalent to a stochastic game \citep{Shapley1953stochastic} where each state represents the previous action of Player~2.
The set of actions in all states is the same as in $A$ and the payoff in state $k$ when Player~1 plays $i$ and Player~2 plays $j$ is $a_{ij}+c s_{kj}$.
In this stochastic game, only Player~2 controls the state transitions.\footnote{The initial state is irrelevant, as in each time step all states are reachable and the payoff from any single stage is negligible compared to the long-term average. For every state $s$, Player~2 can ensure that the game starting at $t=2$ will have initial state $s$.}
This class of games was studied by \cite{filar1980algorithms}, who showed that the value exists and obtained in strategies that depend solely on the current state, not on the entire history of the game, namely \emph{stationary strategies}.
It follows that the value of the repeated game exists as well, and it is obtained in strategies that depend solely on the previous action of Player~2.
We name these strategies too \emph{stationary strategies} and denote the value for each $c$ by $v(c)$.

\begin{definition}
A \emph{stationary strategy} is a strategy that depends in each $t$ on the previous action of Player~2, $j(t-1)$, but not on the rest of the history or $t$ itself.
Hence, a stationary strategy is a vector of $n$ mixed actions, one to follow each possible pure action of Player~2.
\end{definition}

\begin{example}[label=ex_123]{\textbf{Evasion game with uniform switching costs.}}\rm

Consider a situation where a defender (Player~2) needs to evade an attacker (Player~1) with three possible hides and uniform switching costs, so $S=\left(\begin{smallmatrix} 0 & 1 & 1\\ 1 & 0 & 1 \\ 1 & 1 & 0 \end{smallmatrix}\right)$.
The payoff when the attacker misses the defender is $0$, and when the attacker catches the defender depends on the location where he was caught, in the following manner: $A=\left(\begin{smallmatrix} 1 & 0 & 0\\ 0 & 2 & 0 \\ 0 & 0 & 3\end{smallmatrix}\right)$.

Without switching costs ($c=0$), the optimal strategy for Player~2 is $\tfrac{1}{11}(6,3,2)$ and the value is $v=\tfrac{6}{11}$.
For clarity, we denote the actions of Player~2 as $L$ (left), $M$ (middle) and $R$ (right). 

In the equivalent stochastic game, there are three states representing the situation after each of the actions is played.
We denote by $s_i$ the state where Player~2 played action $i \in \{L,M,R\}$ in the previous time period. The last row represents the next state after each column was chosen:
\[s_L:\left(\begin{matrix} 1 & c & c\\ 0 & 2+c & c \\ 0 & c & 3+c \\s_L & s_M & s_R \end{matrix}\right) \qquad 
s_M:\left(\begin{matrix} 1+c & 0 & c\\ c & 2 & c \\ c & 0 & 3+c \\s_L & s_M & s_R  \end{matrix}\right) \qquad 
s_R:\left(\begin{matrix} 1+c & c & 0\\ c & 2+c & 0 \\ c & c & 3 \\s_L & s_M & s_R  \end{matrix}\right) \]
In stationary strategies, each player chooses three mixed strategies, one for each state.
\hfill$\Diamond$
\end{example}

We also consider a simpler class of strategies, which are history-independent.
As discussed in the introduction, these strategies are simpler as they require much less ($n-1$ relative to $n^2-n$) parameters to define.
Moreover, in some applications they are the only one permitted.

\begin{definition}
A \emph{static strategy} is a strategy that plays the same mixed action in each stage, regardless of the history.
Hence, a static strategy is one mixed action played repeatedly.
\end{definition}



Typically, when we limit the players to static strategies, the value may not exist.
This can be understood if we compare the maximin and the minimax of the repeated game in static  strategies.
In terms of the maximin, the static best response of Player~2 to a static strategy can be pure and no switching is needed.
On the other hand, the minimizing strategy of Player~2 is typically mixed and the payoff depends on $c$.
Therefore, and following the literature \citep{schoenmakers2008repeated,rass2014numerical,GDO2020paper} when considering static strategies, the figure of merit we study is the minimax -- the maximal cost that Player~2 can guarantee to pay.
If Player~1 uses static  strategy $x$ and Player~2 plays static  strategy $y$, the payoff can be written in matrix notation as
\begin{equation}\label{eq_g(c)}
g(c)(x,y)=x^{T} A y+ c y^{T} S y,
\end{equation}
and the minimax can be defined as  
\begin{equation}
\tilde{v}(c)=\min\limits_y \max\limits_x g(c)(x,y).
\end{equation}
For simplicity, we sometimes refer to this as the value in static  strategies, but it should be understood as only the \emph{minimax value}.

Our results are presented in two parts.
First, we characterize both value functions.
Next, we bound the difference between them, which represents Player~2's possible loss from playing a static  strategy instead of a stationary strategy (clearly, $\tilde{v}(c)\geq v(c)$).

\noindent \textbf{The weighted average model.}
Another natural way to include switching costs in the payoff is by considering some weighted average between the stage payoff and the switching costs.
This was done by \cite{rass2014numerical} and others \citep{rass2017defending,wachter2018security,GDO2020paper}, who studied a model where players use static  strategies and the payoff is a weighted average of the switching costs and the costs of the game itself: $\alpha x^{T} A y+ (1-\alpha) y^{T} S y$, with $\alpha\in [0,1]$.
Their model is slightly different from our model.
For example, in our work, it is straightforward that as $c$ increases, the situation worsens for Player~2.
However, in their case, when the weight on the switching costs tends to $1$, the value of the game (with switching costs) tends to $0$ and the situation ``improves'' if the value without switching costs was positive.

Nevertheless, the models are equivalent.
Setting $c=\tfrac{1-\alpha}{\alpha}$, we get the same ratio between the two payoffs and, as a result, the games are strategically equivalent.
Hence, our work sheds some theoretical light on their results as well as on general economic situations in which the switching costs are added to the real costs and not averaged with them.

\section{Results}\label{s_results}

In Subsection \ref{s_propv}, we characterize the value function $v(c)$ and the minimax value function in static  strategies $\tilde{v}(c)$.
We benefit from these characterizations in Subsection \ref{s_boundmain}, where we bound the difference between the two for any game with switching costs.
In Subsections \ref{s_smallc} and \ref{s_constantsc}, we find tighter bounds by considering a more restrictive model. 
With the appropriate restrictions on $c$, the bound can even reduce to zero, which implies that in some cases there is a static  strategy that is optimal. 
We also make restrictions on the structure of the switching-cost matrix $S$ by making the off-diagonal elements equal. 

Clearly, if Player~2 has an optimal \emph{pure} action in $A$ without switching costs, then this action is also optimal with switching costs and $v(c)=\tilde{v}(c)=v(0)$ for all $c\geq 0$.
In the rest of the paper, we therefore assume that the value of the game $A$ without switching costs cannot be obtained in pure actions.

\subsection{The Properties of the Value Functions}\label{s_propv}

In this subsection we characterize $v(c)$ where $c \geq 0$, namely the value function in stationary strategies  (\thref{lem:proportiesofvc}). 
We show that there are finitely many segments of the nonnegative real line such that for each segment there exists an optimal stationary strategy of Player~2 that is independent of the exact $c$ (\thref{cor:optimalstationaryispiecewise}).
We also characterize $\tilde{v}(c)$ where $c \geq 0$, namely the minimax value function in static  strategies (\thref{lem:propertiesoftildevc}).
We continue Example \ref{ex_123} to demonstrate these value functions and the corresponding optimal strategies.
We then provide two counterexamples to two natural conjectures.
In Example \ref{ex_tildevc_nonlinear}, we show that  \thref{cor:optimalstationaryispiecewise} does not hold for $\tilde{v}(c)$ in general.
In Example \ref{ex_vcneqtildevc_cyclic}, we show that it is possible for stationary strategies to strictly outperform static  strategies for all $c>0$.

\begin{proposition}\thlabel{lem:proportiesofvc}
For every $c\geq 0$, the game has a value in stationary strategies denoted by $v(c)$.
This function is continuous, increasing, concave, and piece-wise linear.
\end{proposition}


\begin{proof}
	This game is equivalent to a stochastic game where the states correspond to the previous pure actions chosen by Player~2.
	In this stochastic game, only one player controls the transitions, so existence follows from \cite{filar1980algorithms} and \cite{filar1984matrix}.
	
	Moreover, according to \cite{filar1984matrix}, the value of the game is the same as the value of a one-shot matrix game, whose pure actions are the pure stationary actions in the stochastic game (essentially, each pure strategy in the one-shot game is a vector of size $n$ dictating which pure actions to choose at each state).
	The payoffs include the corresponding payoffs from matrix $A$ as well as $c$-dependent payoffs from the transitions.
	Since only Player~2 controls the states, the $c$-dependent part is the same for all entries in the same column and is linear in $c$.
	For example, if the pure action of Player~2 is to alternate between actions $2$ and $1$, the $c$-dependent part of the column is $\tfrac{s_{12}+s_{21}}{2}c$.
	We denote the entry in the $i,j$ place of this one-shot game by $b_{ij}+\beta_j c$.
	
In order to characterize $v(c)$, we look at the value function of this one-shot matrix game. The value of this game is determined by the optimal strategy profiles, which are the same as the Nash-equilibrium strategy profiles. We show that given the support of the Nash-equilibrium strategy profiles, the expected payoff from Player~1's best response is linear in $c$. We also show that the value function is pieced together from these finitely many linear functions, so it is piece-wise linear in $c$.

First let $I_1$ be some subset of the rows and $I_2$ be a subset of the columns of the auxiliary  one-shot game. Fix the cost $c$ to $0$. We check whether Player~2 can make Player~1 indifferent among all the actions in $I_1$, using a completely mixed action over $I_2$, and can make Player~1 prefer them over actions not in $I_1$. Let $k$ be a pure action in $I_1$. We look for a vector $y\in \RR^n$ of Player $2$, such that
\begin{itemize}
\item the support of $y$ is included in $I_2$:$\sum\limits_{j\in I_2} y_j =1$,
\item the support of $y$ is exactly $I_2$: $\forall j\in I_2:\,y_j>0$,
\item Player~1 is indifferent among the actions in $I_1$:
\begin{equation}\label{ind_0}
\forall l\in I_1\setminus\{k\}:\quad\sum\limits_{j\in I_2} y_j b_{kj}= \sum\limits_{j\in I_2} y_j b_{lj},
\end{equation}
\item Player~1 prefers actions in $I_1$ over other actions:
\begin{equation}\label{dom_0}
\forall l\notin I_1:\quad\sum\limits_{j\in I_2} y_j b_{kj} \geq \sum\limits_{j\in I_2} y_j b_{lj}.	
\end{equation}
\end{itemize}
If there exists a solution of this system of equations, we denote by $y(I_1,I_2)$ one solution of this system.  Notice that we do not impose any restriction on the strategy of Player~1.


Furthermore, let $(I_1,I_2)$ be such that $y(I_1,I_2)$ is well defined then $y(I_1,I_2)$ also makes Player~1 indifferent among actions in $I_1$ and prefers actions in $I_1$ over other actions for any $c'\geq0$.  Indeed, let $c'>0$ then Equations (\ref{ind_0}) and (\ref{dom_0}) yields the same equations with the cost:
 Player~1 is indifferent among the actions in $I_1$
\begin{equation}
\forall l\in I_1\setminus\{k\}:\quad\sum\limits_{j\in I_2} y_j(b_{kj}+\beta_j c')= \sum\limits_{j\in I_2} y_j(b_{lj}+\beta_j c'),
\end{equation}
and prefers them over other actions if
\begin{equation}
\forall l\notin I_1:\quad\sum\limits_{j\in I_2} y_j(b_{kj}+\beta_j c')\geq \sum\limits_{j\in I_2} y_j(b_{lj}+\beta_j c').	
\end{equation}
Moreover, let $\sigma$ be a best-response to $y(I_1,I_2)$ then the payoff is a linear function of the cost. Indeed, when Player~2 plays according to $y(I_1,I_2)$, the expected payoff of each row is a linear function of $c$ with the same slope (as shown above), and a higher constant for the $I_1$ rows relative to the others. Hence, when Player~1 best responds, he chooses an action in $\D (I_1)$ and the payoff will be a linear function of $c$, denoted by $l(I_1,I_2)(c)$.\\

Let us now prove that for every $c \geq 0$, there exists $I_1(c)$ and $I_2(c)$ such that the value coincides with $l(I_1(c),I_2(c))(c)$. Fix $c \geq 0$ and a pair of optimal strategy $(x^*,y^*)$. This profile has a corresponding support pair $(I_1(c),I_2(c))$. By definition, $y(I_1(c),I_2(c))$ makes Player~1 indiffferent between his actions, hence the pair $(x^*,y(I_1(c),I_2(c)))$ is a Nash-equilibrium and therefore $y(I_1(c),I_2(c))$ is optimal. It follows that
\[
v(c)=l(I_1(c),I_2(c))(c).
\]
Moreover, the value function $v(\cdot)$ is continuous in $c$ (this is a polynomial game, see \cite{shapley1950basic}), hence the pair $(I_1(c),I_2(c))$ can only change between two pairs inducing different linear functions when these two linear functions are equal. Since they are a finite number of supports, there are finitely many intersections possible and therefore $v(\cdot)$ is piece-wise linear in $c$.

	
	
In addition, since all the entries in $S$ are non-negative, all the coefficients of $c$ (the slopes of the lines) are non-negative, so $v(c)$ is an increasing function of $c$.
	
Finally, the function $v(c)$ is concave.
Let $[c_k,c_{k+1}]$ and $[c_{k+1},c_{k+2}]$ be two segments where $v(c)$ is linear.
For $c^*\in (c_k,c_{k+1})$, let $y(c^*)$ be the optimal action of Player~2 chosen using the above method.
If Player~2 plays $y(c^*)$ regardless of $c$, the payoff is a linear function of $c$ that coincides with the value on $[c_k,c_{k+1}]$.
In the region $[c_{k+1},c_{k+2}]$, the value function must be below this line, so the slope of $c$ must decrease.
Hence, $v(c)$ is a piece-wise linear function with decreasing slope and thus it is concave.\hfill\end{proof}

If on an interval the value function $v(c)$ is linear, then there is a Player~2 strategy that is optimal for this entire interval, irrespective of the precise value of $c$.
A direct result of the piece-wise linearity of $v(c)$ is that the exact value of $c$ does not need to be known to find an optimal strategy, only the segment that contains it matters.
This means that the optimal strategies are robust to small changes of $c$: knowing the exact $c$ is not necessary to play optimally, and it is almost universally unnecessary to adjust the optimal strategy as $c$ changes.
Moreover, this shows that when the model arises from a weighted average of different goal functions \citep{rass2014numerical}, the exact weight given to the switching costs compared to the stage payoff is of smaller significance, if any (as the authors indeed report for static  strategies).

Even if there is no information on $c$, an optimal strategy for Player~2 can still be found in a finite set that depends only on $A$.
This observation can serve as a basis for an algorithm to compute an optimal stationary strategy in the game with switching costs.
Interestingly, this is not true for Player~1.
In general, the optimal strategy for Player~1 depends on the exact $c$ and changes continuously with $c$.

\begin{corollary}\thlabel{cor:optimalstationaryispiecewise}
There exists $K\in \NN$ and $0=c_0<c_1<\ldots<c_K<c_{K+1}=\infty$ such that for each segment $[c_i,c_{i+1}]$ there exists an optimal stationary strategy of Player~2 that is independent of the exact $c$.
\end{corollary}
\begin{proof}
Since the function $v(c)$ is piece-wise linear, the segments are chosen such that in each one, $v(c)$ is linear.
Fix $c^*\in(c_i,c_{i+1})$ and fix $\sigma_{c^*}$ an optimal strategy for Player~2 in the stochastic game.
Suppose Player~2 plays $\sigma_{c^*}$ regardless of $c$ and Player~1 best responds.

When Player~2 plays $\sigma_{c^*}$, he fixes the transition probabilities between states, and these are now independent of $c$.
Hence, the percentage of time spent in each state is constant and independent of $c$.
In each state, all entries in a particular column have the same slope (in state $i$, column $j$ has slope $s_{ij}$), so their mix according to $\sigma_{c^*}$ results in parallel lines (parallel in each state, not necessarily parallel between states).
Player~1 best responds in each state, but since the lines are parallel, his choice does not depend on $c$.
Thus, the payoff in each state is a linear function of $c$, and the total payoff, which is a weighted average of these functions with the percentage of time spent in each state as weights, is a linear function of $c$.
Denote this function by $f_{c^*}(c)$.

Since $v(c)$ is the value, $v(c)\leq f_{c^*}(c)$ for all $c$ in the segment with equality at $c^*$.
Thus, these two lines that intersect once must coincide and $\sigma_{c^*}$ obtains the value in the entire segment.
\hfill\end{proof}

The concavity of $v(c)$ can be intuitively understood as follows.
For a large $c$, it is more costly to switch so Player~2 chooses strategies that incur smaller switching costs despite some loss in the stage game $A$.
Hence, as $c$ increases, the coefficient of $c$ decreases and the rate at which the value increases with $c$ is reduced.
This observation can serve in practice to eliminate from the choice of strategies those with too-high expected per-stage switching costs.

The minimax value in static  strategies exhibits the same properties, except piece-wise linearity.
Here, it can only be established that the value function is semi-algebraic, so there is no version of  \thref{cor:optimalstationaryispiecewise} for this case.
It is quite possible that the optimal strategy depends on $c$, as shown in Example \ref{ex_tildevc_nonlinear}.

\begin{proposition}\thlabel{lem:propertiesoftildevc}
	The minimax value $\tilde{v}(c)$ in static  strategies is a continuous, increasing, and concave semi-algebraic function.
\end{proposition}
\begin{proof}
Let $g(c)(x,y)=x^{T} A y+ c y^{T} S y$ be a function from $\RR_+\times \Delta(\{1,\ldots,m\})\times\Delta(\{1,\ldots,n\})$ to $\RR$.
By definition, 
	\begin{align}\label{eq:tildevdefinedasminimax}
	\tilde{v}(c)=\min\limits_y\max\limits_x g(c)(x,y)=\min\limits_{y \in \Delta(\{1,\ldots,n\})} \left\{\max\limits_{x \in \Delta(\{1,\ldots,m\})} \{x^T Ay\} + cy^T Sy \right\}.
	\end{align}
Note that $\argmax\limits_x g(c)(x,y)$ depends solely on $y$ and not on $c$ (and can always be one of $\{1,\ldots,m\}$).
Since $g(c)(x,y)$ is continuous and since the maximum and the minimum of a continuous function is also continuous, so is $\tilde{v}(c)$.
Moreover, $\tilde{v}(c)$ is semi-algebraic since $g(c)(x,y)$ is a polynomial in each variable.

	\textit{The function $\tilde{v}(c)$ is increasing:}
	Let $0\leq c_1 < c_2$. We show that $\tilde{v}(c_1)\leq \tilde{v}(c_2)$.
	Let $y_{c_2}$ be the $\argmin$ of the minimization in Eq.~\eqref{eq:tildevdefinedasminimax}.
	\begin{align}
	\tilde{v}(c_2)=\max\limits_{x \in \Delta(\{1,\ldots,m\})} \{g(c_2)(x,y_{c_2})\}\geq \max\limits_{x' \in \Delta(\{1,\ldots,m\})} \{g(c_1)(x',y_{c_2})\}\geq \tilde{v}(c_1). \nonumber
	\end{align}
	The equality follows from the definition of $\tilde{v}(c_2)$. 
	The first inequality follows from $c_1 < c_2$ and $S \geq 0$ (the maximizing $x$ is the same as it depends solely on $y_{c_2}$). 
	The last inequality follows from the definition of $\tilde{v}(c_1)$.

	\textit{The function $\tilde{v}(c)$ is concave:}
	Let $0\leq c_1 < c_2$ and let $\beta \in (0,1)$.
	Then
	\begin{eqnarray*}
	\tilde{v}(\beta c_1 + (1-\beta)c_2) 
	&=&\min\limits_{y} \left\{\max\limits_{x} \{x^T Ay\} + (\beta c_1 + (1-\beta)c_2)y^T Sy \right\}  \\
	&=&\min\limits_{y} \left\{\beta \max\limits_x g(c_1)(x,y) +(1-\beta) \max\limits_x g(c_2)(x,y) \right\}
 \\
	 &\geq &	\min\limits_{y} \left\{\beta \max\limits_x g(c_1)(x,y)\right\} +	\min\limits_{y} \left\{(1-\beta) \max\limits_x g(c_2)(x,y) \right\}\\
		&=& \beta \tilde{v}(c_1) + (1-\beta)\tilde{v}(c_2).
	\end{eqnarray*}
The first and last equalities are the definition of $\tilde{v}$. 
The second equality is based on the fact that $\argmax\limits_x g(c)(x,y)$ does not depend on $c$.
The inequality is obtained since minimizing each term individually and then summing results yields a smaller result than first summing and then minimizing the entire expression. 
\hfill\end{proof}

In the following example, we demonstrate how the two values behave as a function of $c$.
Only the results are presented here; the calculations are in Appendix \ref{app_acoe}.

\begin{example}[continues=ex_123]{\textbf{Evasion game with uniform switching costs.}}\rm

\noindent\textbf{The Value of the Game -- $v(c)$}

\textit{Case 1: $c<\tfrac{22}{31}$:}
The optimal strategy for Player~2 in game $A$ is still optimal for this case, and the value is $v(c)=\tfrac{6}{11}+\tfrac{72}{121}c$.

\textit{Case 2: $\tfrac{22}{31}\leq c\leq \tfrac{121}{156}$:}
The switch from $L$ to $R$ is too costly, so action $R$ is removed from state $s_L$ and the optimal strategy becomes $\tfrac{2}{3}(L),\tfrac{1}{3}(M)$.
The optimal strategies in the other states are the same as in the game without switching costs.
The value is $v(c)=\tfrac{156c+198}{319}$.

\textit{Case 3: $c\geq \tfrac{121}{156}$:}
The switch from $L$ to $M$ is too costly, so $s_L$ becomes an absorbing state with payoff $1$.

\smallskip

\noindent\textbf{The Minimax Value in Static Strategies -- $\tilde{v}(c)$}

In this simple case there is no intermediate strategy for Player~2.
He either plays the optimal strategy of game $A$ when switching costs are low enough or he plays the pure minimax strategy of $A$ and never switches when they are high enough.
A direct computation reveals that the cut-off $c$ is $\tfrac{55}{72}$.

\bigskip

\noindent\textbf{Comparison:}
Figure \ref{Fig:FixedVsStationary} shows the static  solution $\tilde{v}(c)$ (dashed) and the stochastic game solution $v(c)$ (solid) in the region where they differ.
The maximal difference is found at $c=\tfrac{55}{72}$ and is equal to $\tfrac{37}{2088}$. \hfill$\Diamond$
\end{example}

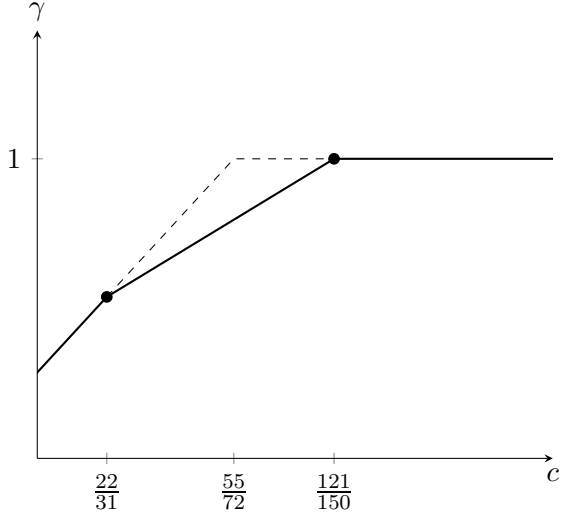
\begin{figure}
\begin{center}
\begin{tikzpicture}

\begin{axis}[axis x line=middle, axis y line=middle, xmin=0.68,  xmax=0.9, ymin=0.93, ymax=1.03, xtick={22/31, 55/72, 121/150}, xticklabels={$\tfrac{22}{31}$,$\tfrac{55}{72}$,$\tfrac{121}{150}$}, ytick={6/11, 1}, yticklabels={$\tfrac{6}{11}$,$1$}, xlabel={$c$}, ylabel={$\gamma$},  x label style={at={(axis description cs:1,0)},anchor=north},    y label style={at={(axis description cs:0,1)},anchor=south}]
	\addplot[mark=*, only marks] coordinates {(0,6/11) (22/31,30/31) (121/150,1)}; 
	\addplot[thick] coordinates{(0,6/11) (22/31,30/31) (121/150,1) (1,1)};
	\addplot[dashed] coordinates{(22/31,30/31) (55/72,1) (121/150,1)};
	
\end{axis}
\end{tikzpicture}
\caption{The value in optimal (stationary) strategies $v(c)$ (solid) and the minimax value in static  strategies $\tilde{v}(c)$ (dashed) of the evasion game from Example \ref{ex_123}.}
\label{Fig:FixedVsStationary}
\end{center}
\end{figure}

With additional assumptions, the result of \thref{lem:propertiesoftildevc} can be strengthened.
In particular, in the common case where all switches are equally costly (uniform switching costs), the result can be strengthened to piece-wise linearity (\thref{lem:tildevc_pw_linear_constantS}).
In the latter case, the set of optimal static  strategies can be fully characterized and it is finite.
This is not true in the general case, where the only conclusion that can be made (see the remarks after \thref{lem:tildevc_pw_linear_constantS}) is that when playing an optimal static  strategy, Player~2 keeps Player~1 indifferent in game $A$ between at least two of his pure actions.

\begin{example}\textbf{A game where $\tilde{v}(c)$ is not piece-wise linear} \rm \thlabel{ex_tildevc_nonlinear}

Suppose $A=\left(\begin{smallmatrix} -200 & 1 & 200\\ 200 & 1 & -200 \end{smallmatrix}\right)$
and $S=\left(\begin{smallmatrix} 0 & 1 & 100\\ 1 & 0 & 1 \\ 100 & 1 & 0 \end{smallmatrix}\right)$.
It is easy to verify that for $c\leq \tfrac{1}{98}$ the optimal strategy (static  and stationary) is the optimal strategy of game $A$ (choosing the first and third columns with equal probability) that leads to a payoff of $\tilde{v}(c)=50c$, and that for $c\geq \tfrac{1}{2}$ the optimal strategy (static  and stationary) is the pure minimax with a payoff of $\tilde{v}(c)=1$.
For $c\in\left[\tfrac{1}{98},\tfrac{1}{2}\right]$, the optimal static  strategy of Player~2 is $[p(L),1-2p(M),p(R)]$ where $p=\tfrac{1-2c}{192c}$, keeping Player~1 indifferent between his two actions.
In this case, $\tilde{v}(c)=1-\tfrac{(1-2c)^2}{192c}$, which is not linear in $c$.
\hfill$\Diamond$
\end{example}

\thref{cor:optimalstationaryispiecewise} suggests that the optimal stationary strategy in the first segment (which includes $c=0$) is optimal even without switching costs, so it must be comprised of optimal strategies in game $A$.
If there is only one optimal strategy in the game (as in Example \ref{ex_123}), then the optimal stationary strategy is actually static  and in that segment, $v(c)=\tilde{v}(c)$.
However, in general, this may not be the case.
The following example shows that it is possible that stationary strategies will always outperform static strategies (except, of course, at $c=0$).

\begin{example}\textbf{A game with $\tilde{v}(c)>v(c)$ for all $c>0$.} \rm
\thlabel{ex_vcneqtildevc_cyclic}

Suppose $A=\left(\begin{smallmatrix} 1 & 0 & 1 & 0\\ 0 & 1 & 0 & 1 \end{smallmatrix}\right)$
and $S=\left(\begin{smallmatrix} 0 & 0 & 1 & 1\\ 1 & 0 & 0 & 1 \\ 1 & 1 & 0 & 0 \\ 0 & 1 & 1 & 0\end{smallmatrix}\right)$.
The value of $A$ is $0.5$ and is achieved by any mixed action of Player~2  choosing the set of even columns w.p. $0.5$.
Switching-cost matrix $S$ is such that, based on the action of Player~2 in the previous time period, the only costless actions are the same action and the action to the right of the previous action.
Due to this cyclical nature of $S$, Player~2 can achieve the value of game $A$ with a stationary strategy, by playing the previous action w.p. $0.5$ and the action to the right w.p. $0.5$.
Clearly, any static  non-pure strategy will contain a $c$ component, and any pure static strategy  will result in $\tilde{v}(c)=1$.
To conclude, $\tilde{v}(c)>v(c)=0.5$ for all $c>0$.
\hfill$\Diamond$
\end{example}

\subsection{A Uniform Bound on the Value Differences in Static  and Stationary Strategies}\label{s_boundmain}

Our goal is to compare $v(c)$ and $\tilde{v}(c)$ and evaluate any loss $\tilde{v}(c)-v(c)$ due to using a static  strategy instead of a stationary strategy.
Our first bound is very general and does not rely heavily on game-specific assumptions (\thref{thm:uniformloss}).
Then we improve the bound by taking into account additional properties of the game, such as the multiplicity of equilibria in $A$ and the structure of $S$ (\thref{obs_symS}).
We start by showing that in the simplest set of games, $2\times 2$ games, the optimal stationary strategy is always a static  strategy (\thref{prop_2x2}).
This is a generalization of Theorems 4.1, 4.2 and Corollary 4.1 in \cite{schoenmakers2008repeated}.

\begin{proposition}\thlabel{prop_2x2}
Let $A$ be a $2\times 2$ zero-sum game.
For every $S$ and every $c$, there exists an optimal stationary strategy in switching-costs game $(A,S,c)$ that is a static strategy.
Thus, $\tilde{v}(c)=v(c)$.
\end{proposition}
\begin{proof}
Let $A=\left(\begin{smallmatrix} \alpha & \beta \\ \gamma & \delta \end{smallmatrix}\right)$ be a zero-sum $2\times 2$ game and $S=\left(\begin{smallmatrix} 0 & s_{12} \\ s_{21} & 0 \end{smallmatrix}\right)$ the switching-costs matrix for Player~2.
We assume that the value in $A$ without switching costs cannot be obtained using pure strategies (otherwise the discussion is trivial), so w.l.o.g., $\alpha>\beta$, $\alpha >\gamma$, $\delta>\beta$, and $\delta>\gamma$.
Let $p\in (0,1)$ be the probability of playing the left column (action $L$) in $A$ that achieves the value (a direct computation shows that only one such $p$ exists).
As in the examples, we denote by $s_i$ the new state when the previous action of Player~2 was $i\in\{L,R\}$.

Based on \thref{lem:proportiesofvc} and the ACOE method (Appendix \ref{app_acoe}), the optimal stationary strategy for Player~2 is such that in each state Player~1 is either indifferent between his two actions or plays purely one of them.
Since the same switching costs and continuation payoffs are added to all entries of the same column of each $s_i$ (see Eq.~\eqref{eq:two_one_shot_for2x2games}), they do not affect Player~1's choice of rows.
Thus, in each state either Player~2 plays purely too or he plays the mixed action $p$.
All the possible combinations of these actions are summarized in Table \ref{tbl_for2x2games}.
For example, in case I he plays $p$ in both states, whereas in case IV he plays $p$ only at state $s_R$, while after playing $L$ (and reaching $s_L$) he always plays purely action $R$.
In cases II and III he is absorbed in states $s_L$ or $s_R$ (resp.), so the action in the other state is irrelevant.

\begin{table}[]
\centering
\begin{tabular}{l|l|l|}
\textnumero    & $s_L$ & $s_R$ \\ \hline
I   & $p$    & $p$    \\ \hline
II  & $L$    & $?$    \\ \hline
III & $?$    & $R$    \\ \hline
IV  & $R$    & $p$	  \\ \hline
V	& $p$	   & $L$  \\ \hline
\end{tabular}
\caption{All the possible optimal stationary strategies.}\label{tbl_for2x2games}
\end{table}

Note that cases I, II, and III represent a static  strategy, so to show that $v(c)=\tilde{v}(c)$ we must show that the strategies presented in cases IV and V cannot be optimal.
Assume by contradiction that the strategy presented in case IV is indeed optimal.
Let $\kappa$ be the continuation payoff after playing $R$ (the continuation payoff after playing $L$ is normalized to $0$).
When we add the switching costs and the continuation payoffs to the games, we obtain two one-shot games:
\begin{equation}\label{eq:two_one_shot_for2x2games}
s_L: \left(\begin{smallmatrix} \alpha & \beta+\kappa+cs_{12} \\ \gamma & \delta+\kappa+cs_{12} \end{smallmatrix}\right) \qquad s_R: \left(\begin{smallmatrix} \alpha+cs_{21} & \beta+\kappa \\ \gamma+cs_{21} & \delta+\kappa \end{smallmatrix}\right) \nonumber
\end{equation}

Let us consider state $s_L$. 
By assumption on the coefficients, $\alpha > \gamma$ and $\delta+\kappa+cs > \beta +\kappa+cs$. Since it is optimal to play purely $R$, we necessarily have
\[
\alpha \geq \beta+\kappa+cs_{12}  \text{ or } \gamma \geq  \delta+\kappa+cs_{12}.
\]
If only the first equation is true, the equilibrium is mixed and $R$ is not optimal.
If only the second equation is true, then there is a contradiction:
\[
\alpha> \gamma \geq \delta+\kappa+cs >\beta+\kappa+cs >\alpha.
\]
It follows that both equations are true, i.e. the right column dominates the left including the continuation payoffs (this is only true in $2\times 2$ games) and in particular $\gamma\geq \delta + cs_{12} + \kappa$.

In state $s_R$, however, the optimal strategy is mixed and therefore such dominance is not possible.
If the direction of the  original inequalities remains correct, then $\delta + \kappa \geq \gamma + cs_{21}$.
It follows that $s_{12}+s_{21}\leq 0$, which is a contradiction.
The other option is that, since in $s_R$ both inequalities changed direction, a contradiction can be constructed using the first row.

A similar approach can be used for case V and the proof is complete.\hfill\end{proof}

Thus, in the $2\times 2$ case, there are two candidates for optimal stationary strategy: either playing the optimal strategy of $A$ regardless of the previous action, or playing the pure minimax action of $A$.
Both are also static  strategies.
It is straightforward to calculate the value in both cases and to determine the cut-off value of $c$ above which playing the pure minimax becomes profitable.

When either of the players has more than $2$ actions in $A$, the value may differ under static  and stationary strategies.
The following Proposition bounds this difference.
Naturally, since this Proposition is general, the resulting bound is not easy to calculate or interpret.
However, in the Theorems that follow, we add more assumptions to obtain a bound in a simple analytic manner.

\begin{proposition}\thlabel{thm:uniformloss}
Player~2's loss from limiting himself to static  strategies, $\tilde{v}(c)-v(c)$, is bounded by
\begin{equation}\label{eq:thm:loss_formula}
\Delta(c) = \tfrac{1}{2}(1-v(c))+\tfrac{c}{4} \Xi+\tfrac{c}{4}M,
\end{equation}
where 
\begin{itemize}
\item $M=\max\limits_{y\in \Delta(\{1,...,n\})} y^{T} S y$  is the maximal switching cost that Player~2 can pay,
\item $\Xi=\max\limits_{j,k} |S_{jk}-S_{kj}|$ is a measure of the asymmetry of the switching-cost matrix.
\end{itemize}
\end{proposition}

\begin{proof}
By definition of value $v(c)$, there exists a state $j\in S$ and a mixed action $y\in \Delta (J)$ such that for all $i \in I$ we have that the cost at $j$ by playing $y$ is less than $v(c)$ (otherwise, the payoff of Player~1 in each state can exceed $v(c)$ and thus can be the total value).

Consider the static  strategy $\widetilde{y}$ of Player~2, that with probability $0.5$ chooses action $j$ and otherwise chooses an action according to mixed action $y$. We know that if Player~2 plays a static  strategy then Player~1 has a static  strategy best response. 
Let $x$ be this best response.
Then, replacing $\tilde{y}$ by its definition, we have
\begin{align}
g(c)(x,\tilde{y})&=x^{T} A \tilde{y}+ c \tilde{y}^{T} S \tilde{y} \nonumber\\
& = \tfrac{1}{2} x^{T} A j + \tfrac{1}{2} x^{T} A y+ c \tfrac{1}{4} j^TS j+c \tfrac{1}{4} j^TS y+c \tfrac{1}{4} y^{T}S j  +c \tfrac{1}{4} y^{T}S y . \nonumber
\end{align}
Bounding $y^{T}S j$ from above with $jSy+\Xi$ and reorganizing leads to
\begin{align}
g(c)(x,\tilde{y}) &\leq \tfrac{1}{2} \left( x^{T} A y+ c j S y \right)+ c \tfrac{1}{4} \Xi+ \tfrac{1}{2} x^{T} A j+ c\tfrac{1}{4} j^TS j+c \tfrac{1}{4} y^{T}S y . \nonumber
\end{align}
We can now replace the first part of the bound by $v(c)$.
Recall that $A$ is normalized so that $\max\limits_{i,j} a_{ij}=1$ and $S$ defined such that $s_{ii}=0$ for every $i\in \{1,\ldots,n\}$ so:
\begin{align}
g(c)(x,\tilde{y}) &  \leq \tfrac{1}{2}v(c)+ \tfrac{c}{4} \Xi+ \tfrac{1}{2}+ \tfrac{c}{4} M. \nonumber
\end{align}
It follows from the definition of $\tilde{v}(c)$ that
\begin{equation}
\tilde{v}(c) \leq \tfrac{1}{2} v(c)+ \tfrac{c}{4} \Xi+\tfrac{1}{2}+\tfrac{c}{4}M. \nonumber
\end{equation}
The bound is obtained by subtracting $v(c)$ from both sides:
\begin{equation}
\tilde{v}(c) -v(c) \leq \tfrac{1}{2}(1-v(c))+\tfrac{c}{4} \Xi+\tfrac{c}{4}M=\Delta(c). \nonumber
\end{equation}
\hfill\end{proof}

In \thref{thm:uniformloss}, the bound depends on $\Xi$, which measures the symmetry of the switching-cost matrix.
Naturally, the symmetry of $S$ has no effect on the minimax value in static  strategies.
Recall that $g(c)(x,y)$ from Eq.~\eqref{eq_g(c)} expressed the expected payoff of Player~2 when the players play a static  strategy pair $(x,y)$.
When the expression for $g(c)(x,y)$ is expanded, each product $y_i y_j$ is multiplied by $(s_{ij}+s_{ji})$.
Hence, when solving for $\tilde{v}(c)$, only the sum of symmetric off-diagonal elements of $S$ matters, not their exact values.
This, for example, is the reason why \cite{GDO2020paper} assumed that $S$ is a symmetric matrix.

A stationary strategy, on the other hand, can take any asymmetry into account, yielding different results.
This can easily be seen in Example \ref{ex_vcneqtildevc_cyclic}.
If we change $S$ into a symmetric matrix while keeping the sums $s_{ij}+s_{ji}$ at their current values, the result will be a matrix where every mixed stationary strategy has a $c$-component, whereas in the non-symmetric case, the optimal stationary strategy kept $v(c)=0.5$ regardless of $c$.

\begin{theorem}\thlabel{obs_symS}
If matrix $S$ is symmetric, Player~2's loss from limiting himself to static  strategies, $\tilde{v}(c)-v(c)$, is bounded by
$\Delta(c) =\tfrac{1}{2}(1-v(c))+\frac{c}{4}M$
\emph{(}where $M$ is defined in \thref{thm:uniformloss}\emph{)}.
This formula implies the following bound
\begin{equation}\label{eq:SymCase:morethanhalf}
\tfrac{1+cM-\tilde{v}(c)}{1+cM-v(c)} \geq \frac{1}{2}.
\end{equation}
\end{theorem}
\begin{proof}
The first bound is a direct corollary from \thref{thm:uniformloss}, setting $\Xi=0$. The second part is obtained by relaxing the bound and reorganizing. By \thref{thm:uniformloss}, we obtain
\begin{align*}
\tilde{v}(c) -v(c) \leq \tfrac{1}{2}(1-v(c))+\tfrac{c}{4}M \leq \tfrac{1}{2}(1+cM-v(c)).
\end{align*}
We can deduce that
\begin{align*}
1+cM-\tilde{v}(c)& =1+cM-v(c)+(v(c)-\tilde{v}(c)),\\
&  \geq (1+cM-v(c))- \tfrac{1}{2}(1+cM-v(c)),\\
&= \tfrac{1}{2}(1+cM-v(c)).
\end{align*}
\hfill\end{proof}

The reorganization of the bound to the fraction form presented in Eq.~\eqref{eq:SymCase:morethanhalf} has an intuitive interpretation: by using static  strategies, Player~2 loses at most half of what he could have obtained with a stationary strategy.
Had we done the analysis from the point of view of the maximizing player, this Equation and the intuition would be further simplified into: ``with static  strategies the player guarantees at least half of the value from stationary ones''. 
This sets the ``price of being static '' at $\tfrac{1}{2}$ when $S$ is symmetric.
Such a calculation is possible in other settings too, such as the ones from  \thref{prop:genquart}, where the ``price of being static '' is $\tfrac{1}{4}$.

It is important to note, however, that this approach is not always possible.
In the following example, when $S$ is asymmetric the stationary strategy yields a worse result than when $S$ is symmetric, and in an extreme case -- the optimal stationary strategy is actually static.

\begin{example} \textbf{Rock-Paper-Scissors (RPS) with non-symmetric $S$.} \rm \thlabel{ex_RPS}

Consider the RPS game with $A=\left(\begin{smallmatrix} 0 & 1 & -1\\ -1 & 0 & 1 \\ 1 & -1 & 0\end{smallmatrix}\right)$ and $S=\left(\begin{smallmatrix} 0 & 1+\D & 1+\D\\ 1-\D & 0 & 1+\D \\ 1-\D & 1-\D & 0 \end{smallmatrix}\right)$, where $-1\leq\D\leq 1$.
The static  solution remains the same regardless of $\D$, and dictates choosing all actions with equal probability if $c\leq 1.5$ and purely playing one of the actions for larger $c$s.

The stationary strategy for $\D=0$ is to choose all actions with equal probability for $c\leq 1$ and to play purely for $c\geq 2$.
For $c\in[1,2]$, the strategy involves cyclic behavior: after $L$ never play $R$, after $M$ never play $L$, and after $R$ never play $M$.
In all cases, the previous action is replayed w.p. $\tfrac{2}{3}$.

For $\D=\tfrac{1}{3}$, the optimal strategy is $\left(\tfrac{1}{3},\tfrac{1}{3},\tfrac{1}{3}\right)$ for $c\leq \tfrac{9}{8}=1.125$.
For a larger $c$, one action is never played in each state, as discussed in the $\D=0$ case.
This is true until $c=\tfrac{9}{5}$.
From then on, the optimal strategy is to play purely.

For $\D=1$, the optimal stationary strategy is in fact the static  one, and the two cases coincide.
All the cases are shown in Figure \ref{Fig:nonSymRPS}.
\begin{figure}
\begin{center}
\begin{tikzpicture}

\begin{axis}[axis x line=middle, axis y line=middle, xmax=2.5, ymin=0, ymax=1.1, xlabel={$c$}, ylabel={$\gamma$},  x label style={at={(axis description cs:1,0)},anchor=north},    y label style={at={(axis description cs:0,1)},anchor=south},  xtick={1, 1.5, 2}, xticklabels={$c_0$,$c^*$,$c_p$}, ytick={0, 1}, yticklabels={$v$, $\bar{v}$}
]
	\addplot[mark=*, only marks] coordinates {(0,0) (1,2/3) (2,1)}; 
	\addplot[thick] coordinates{(0,0) (1,2/3) (2,1) (2.5,1)};

	\addplot[dashed] coordinates{(1,2/3) (1.5,1) (2,1)};
		
	\addplot[dotted] coordinates{(9/8,3/4) (9/5,1)};

\end{axis}
\end{tikzpicture}
\caption{Different value functions for non-symmetric RPS.
Solid: $\D=0$. Dashed: $\D=1$, static  strategy. Dotted: $\D=\tfrac{1}{3}$.}\label{Fig:nonSymRPS}
\end{center}
\end{figure}
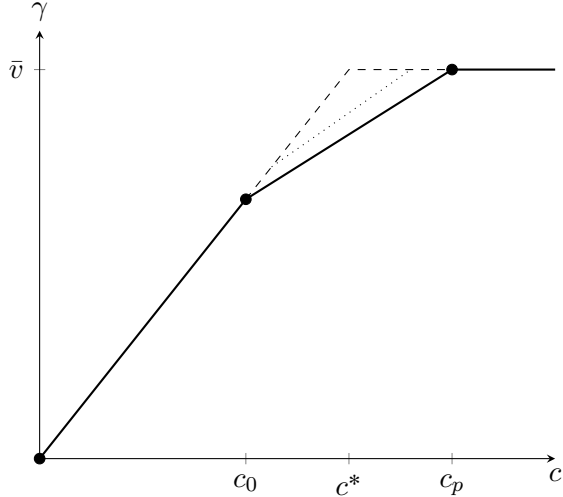
\hfill$\Diamond$
\end{example}

The bound from \thref{thm:uniformloss} is a general bound without making any assumptions on game $A$ or matrix $S$.
As shown in \thref{obs_symS}, tighter bounds can be tailored  to special cases.

In Subsection \ref{s_smallc} we show that in certain cases, when $c$ is either small or large,  the optimal strategies are static. 
This implies that the two value functions coincide on these segments and the bound is $0$.
	
In Subsection \ref{s_constantsc}, we make further assumptions on the matrix $S$.
While the diagonal elements in $S$ are all zeros, the off-diagonal elements are ones.
This means that all switches are equally costly, as suggested by \cite{lipman2000switching,lipman2009switching}.
In this case, not only $v(c)$ but also $\tilde{v}(c)$ is piece-wise linear, which strengthens the robustness result and simplifies the process of finding the optimal static  strategy.
The search domain narrows down to a finite (although large) predetermined set of mixed actions.

\subsection{Optimality of a static  strategy for Small and Large Values of $c$}\label{s_smallc}

In this section we tackle two possible assumptions affecting the two extremes of $c$.
We first show how to identify the case where, for small values of $c$, switching costs do not play a role and the optimal stationary strategy is a static  one.
We next assume that there are no free switches, i.e., that all non-diagonal elements of $S$ are non-zero.
In this case, again for large enough $c$, the optimal stationary strategy is a static  one, and moreover, pure.

We start with the small $c$ case.
In most models, switching costs are either ignored or considered irrelevant, so comparing a case with no switching costs ($c=0$) to one with small switching costs is important as a robustness check on this claim.
\cite{lipman2009switching} provide one such comparison in discussing the continuity of equilibrium payoffs and equilibrium strategies when switching costs converge to zero.
Here, we provide another comparison: between the optimal stationary and static  strategies when $c$ is very small.
We have already established that in some cases (Example \ref{ex_123}) small switching costs do not matter and the optimal stationary strategy is also a static  one, while in other cases (Example \ref{ex_vcneqtildevc_cyclic}) there is a difference between the two for every $c>0$.
Hence, we are interested in finding conditions enabling us to distinguish between the two types of cases.

Clearly, for $c=0$, value $v(0)$ and the minimax value in static  strategies $\tilde{v}(0)$ coincide.
Let $\ubar{c}$ be the maximal $c$ such that for every $c\in[0,\ubar{c}]$ the value is attained by a static  strategy, so $v$ and $\tilde{v}$ still coincide.
As \thref{prop:ubarc} shows, the exact value of $\ubar{c}$ can be found in polynomial time.\footnote{With additional assumptions, this can be theoretically studied for particular classes of games. For example, \cite{schoenmakers2008repeated} showed in their Corollary 6.1 that if $A$ is a square matrix, the unique optimal strategy for both players without switching costs is completely mixed and all off-diagonal entries in $S$ are equal (as in our \S \ref{s_constantsc}), then $\ubar{c}>0$.}
If $\ubar{c}>0$, it immediately reveals the optimal strategy in the entire $[0,\ubar{c}]$ region, both static  and stationary.
This optimal strategy in the game with switching costs is repeatedly playing an optimal mixed action in game $A$ that incurs the lowest switching costs according to $S$.

\begin{proposition}\thlabel{prop:ubarc}
Let $\ubar{c}$ be the maximal $c$ such that $v(c)=\tilde{v}(c)$  for every $c\in[0,\ubar{c}]$.
If $\ubar{c}>0$, then for every $c\in[0,\ubar{c}]$ the value is attained by the static  strategy $y^*=\argmin\limits_{y\in\mathcal{A}} y^T S y$.
This $\ubar{c}$ can be found in polynomial time.
\end{proposition}
\begin{proof}
First, note that $y^*=\argmin\limits_{y\in\mathcal{A}} y^T S y$, the optimal strategy for Player~2 in game $A$ with the lowest average switching costs is well defined: $\mathcal{A}$ is a compact set, $y^T S y$ is continuous, so it has a minimum and $y^*$ can be taken as any of them.
We assume that $y^*$ is not pure, otherwise the entire discussion is trivial.

We still need to find $\ubar{c}$.
Let $I_1$ be the largest support of optimal actions of Player~1 in $A$ and $x_0\in I_1$ (arbitrarily chosen).
Use ACOE (see Appendix \ref{app_acoe}) with undiscounted payoff $\gamma(c)= v + c {y^*}^T S y^*$ and actions $x_0,y^*$ in each state.
This results in an $n\times n$ linear equation, where the unknowns are the continuation payoffs, their coefficients are the elements of $y^*$, and the free elements are linear functions of $c$.
Solve it to find the continuations payoffs (they are linear functions of $c$).

Then, plug in the continuation payoffs to each state and consider each state as a one-shot game.
Since both the switching costs and the continuation payoffs are the same in each column, as in the proof of \thref{lem:proportiesofvc}, they cancel out. 
Thus when Player~2 plays $y^*$, he makes Player~1 indifferent among all the actions in $I_1$ and prefer them over all the other actions.

Next, check whether in each such game, $y^*$ can actually be played by Player~2 in equilibrium, i.e. check whether Player~1 has a mixed action over $I_1$ that keeps Player~2 indifferent among all the columns for which $y^*_i>0$ and still preferring them over the ones where $y^*_i=0$ (this time, including the switching costs and the continuation payoffs!).
Each such test consists of solving another (order of) $n\times n$ linear equation (and verifying some inequalities) and results in an upper bound on $c$ for which the solution is indeed a distribution over the rows.
The minimum of all these $n$ upper bounds is the required $\ubar{c}$.

In this proof, one solves one $n\times n$ linear equation, then $n$ linear equations with $m$ variables and $n$ equations.
Since the complexity of solving linear equations is polynomial in the dimensions, so does the entire algorithm.
When $m$ is of the order of $n$, this is bounded by $O(n^4)$.
\hfill\end{proof}

Recall that we assume that the diagonal of $S$ is zero, while the off-diagonal elements are non-negative and for normalization, the minimal non-zero element is $1$.
Now we also assume that there are no free switches, i.e., all off-diagonal elements of $S$ are strictly positive.
In this case, increasing $c$ increases the cost of playing mixed strategies and, at some point, it will be optimal to play purely.
The cutoff after which the optimal stationary strategy is pure is denoted by $\bar{c}$.
The following proposition proves its existence and generalizes Theorem 6.3 in \cite{schoenmakers2008repeated}.

\begin{proposition}\thlabel{prop:barc}
Let $S$ be a matrix where all off-diagonal elements are strictly positive.
Then there exists $\bar{c}$ s.t. Player~2's optimal strategy for every $c\geq\bar{c}$ is the pure minimax strategy in the matrix game $A$, but not for $c<\bar{c}$.
Hence, $v(c)= \tilde{v}(c)=\bar{v}$ for $c\geq \bar{c}$, where $\bar{v}$ is the minimax of game $A$ in pure actions.
\end{proposition}
\begin{proof}
In each state, when $c\to\infty$, switching a pure action is strictly dominated by replaying it (alternatively, when $c\to\infty$, in the matrix from \cite{filar1984matrix}, all pure actions that represent switches become dominated since $c$ is multiplied by positive constants).
Hence, for $c$ large enough, Player~2 never changes his action and plays the same pure action for all $t$.
Under this restriction, the optimal strategy is to play the pure minimax action of matrix game $A$.
Clearly, this is also a static  strategy.
\hfill\end{proof}

The case of large $c$ is interesting, as it represents a situation where the switching cost is higher than the possible one-stage gain.
For example, the profit of a firm that adjusts its price to dominate a market can be much smaller than the adjustment cost, provided that the other firms in the market respond quickly enough.
\cite{chak90characterizations} studied this scenario for supergames and showed how such high switching costs affect the set of possible equilibria and the optimal strategies, which must include only finitely many action switches.

Although we have no efficient algorithm to calculate $\bar{c}$, an upper bound for it can be found in the following manner.
A sufficient condition for the optimal stationary strategy to be pure is that one pure action of the one-shot game presented in \cite{filar1984matrix} strictly dominates all others (see the proof of \thref{lem:proportiesofvc}).
Hence, every mixed entry in the matrix should suffice $\bar{v}\leq b_{ij}+\beta_j c$.
Assuming $\beta_j>0$, we get $c\geq \tfrac{\bar{v}-b_{ij}}{\beta_j}$.
When this condition is met for all $i,j$, we get $\tilde{v}(c)=\bar{v}$, hence $\bar{c}\leq \max\limits_{i,j} \left(\tfrac{\bar{v}-b_{ij}}{\beta_j}\right)$.

Ideally, we would like to have the exact value of $\bar{c}$ or the smallest possible upper bound (this will more accurately provide the difference between $v(c)$ and $\tilde{v}(c)$).
For example, checking all the values in the matrix and finding the maximal for the above expression would be a better algorithm than what follows.
However, given the large matrix (order of $n^m\times n^n$), computation would be hard and time-consuming.
Instead, we can maximize the expression element-by-element -- requiring that $\beta_j$ and $b_{ij}$ be minimal but not necessarily for the same $i,j$.

By normalization, the lowest off-diagonal switching cost is $1$.
Since $\beta_j$ is the average switching cost  per stage, we take it to be the minimum, which is $1$.
Similarly, for $b_{ij}$ we can choose the lowest entry in $A$.
An upper bound on $\bar{c}$ can be written as
\begin{equation}
\hat{\bar{c}}=\bar{v}-\min\limits_{i,j} a_{ij}=\bar{v},
\end{equation}
since $A$ was normalized s.t. $\min\limits_{i,j} a_{ij}=0$.
In all the following calculations and formulas, if the exact value of $\bar{c}$ is unknown, it is possible to use $\hat{\bar{c}}$ remembering that $\hat{\bar{c}}\geq \bar{c}$ and hence the bound will be less precise (larger).

Subsequently, the only region where stationary strategies might out-perform static  ones is $(\ubar{c},\bar{c})$.
As the following example shows, however, a static  strategy could be optimal for some $c$s in this region.
\begin{example}\textbf{A game where $v(c)$ and $\tilde{v}(c)$ intersect in $(\ubar{c},\bar{c})$} \rm \thlabel{ex_vcmeetstildevc}

Suppose $A=\left(\begin{smallmatrix} 1 & 0 & 0\\ 0 & 1 & 0 \\ 0 & 0 & 1 \end{smallmatrix}\right)$ and $S=\left(\begin{smallmatrix} 0 & 1 & 2\\ 1 & 0 & 3 \\ 2 & 3 & 0 \end{smallmatrix}\right)$.
For $c<\tfrac{1}{6}$, the optimal stationary strategy is the static  strategy that chooses each action w.p. $\tfrac{1}{3}$, regardless of the history.
Starting from $c=\tfrac{1}{6}$, it is too costly to play $R$ after $M$, so the optimal stationary strategy is to play $[0.5(L), 0.5(M)]$ in $s_M$ (and each action w.p.\ $\tfrac{1}{3}$ in the two other states).
In this region, the stationary strategy outperforms static  strategies.

Next, at $c=\tfrac{3}{13}$, it is also too costly to play $R$ in $s_L$ and this action is never played.
The optimal stationary strategy is the static  strategy that plays $[0.5(L),0.5(M)]$ at every stage.
This remains true until $c$ is high enough to make switching too costly and the optimal strategy is playing a pure action which leads to $v(c)=\tilde{v}(c)=1$. \hfill$\Diamond$
\end{example}

In the above example, $v(c)=\tilde{v}(c)$ before $\bar{c}$, where the value functions flatten.
As outside $(\ubar{c},\bar{c})$ the value functions are equal, the only region of interest is $(\ubar{c},\bar{c})$, and the following theorem bounds the difference in payoff functions there.

\begin{theorem}\thlabel{prop:loss}
Player~2's loss from limiting himself to static  strategies, $\tilde{v}(c)-v(c),$ is bounded by
\begin{equation}\label{eq:prop:loss_formula}
\Delta(c) =
 \begin{cases}
   (c-\ubar{c})\tfrac{\bar{c}-c_0}{\bar{c}-\ubar{c}}s        & \text{if } c \in (\ubar{c},c_0), \\
   (\bar{c}-c)\tfrac{c_0-\ubar{c}}{\bar{c}-\ubar{c}}s        & \text{if } c \in (c_0,\bar{c}), \\
   0        & \text{otherwise,}
  \end{cases}
\end{equation}
where $s={y^*}^T S y^*$, $c_0=\tfrac{\bar{v}-v}{s}$ and $\ubar{c}$ and $y^*$ are defined in \thref{prop:ubarc}, and $\bar{c}$ in \thref{prop:barc}.
\end{theorem}

\begin{proof}
Define $s={y^*}^T S y^*$ ($y^*$ is defined in \thref{prop:ubarc}).
From concavity, $v(c)$ lies above the line that connects $(\ubar{c},v+\ubar{c} s)$ and $(\bar{c},\bar{v})$ (denoted by $l$).
Function $\tilde{v}(c)$ must be below lines $\bar{v}$ (corresponding to a pure static  strategy) and $v+c s$ (corresponding to static  strategy $y^*$).
These lines meet at $c_0=\tfrac{\bar{v}-v}{s}$.
Hence, for $c<c_0$ the difference between the value in static  and in stationary strategies is at most the difference between line $v+cs$ and $l$, and for $c>c_0$ it is the difference between $\ubar{v}$ and $l$ (see Fig. \ref{Fig:BoundOnBeingFixed}).
A direct computation reveals that these differences are as in Eq.~\eqref{eq:prop:loss_formula}.\hfill\end{proof}

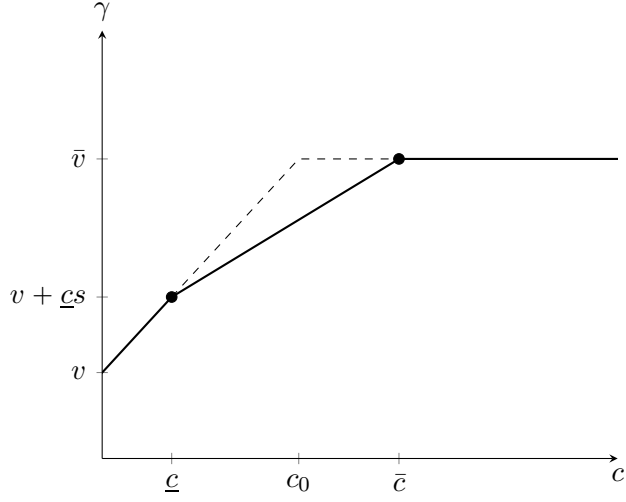
\begin{figure}
\begin{center}
\begin{tikzpicture}

\begin{axis}[axis x line=middle, axis y line=middle, xmin=0.68,  xmax=0.9, ymin=0.93, ymax=1.03, xtick={22/31, 55/72, 121/150}, xticklabels={$\ubar{c}$,$c_0$,$\bar{c}$}, ytick={6/11+0.68*72/121,30/31, 1}, yticklabels={$v$,$v+\ubar{c}s$,$\bar{v}$}, xlabel={$c$}, ylabel={$\gamma$},  x label style={at={(axis description cs:1,0)},anchor=north},    y label style={at={(axis description cs:0,1)},anchor=south}]
	\addplot[mark=*, only marks] coordinates {(0,6/11) (22/31,30/31) (121/150,1)}; 
	\addplot[thick] coordinates{(0,6/11) (22/31,30/31) (121/150,1) (1,1)};
	\addplot[dashed] coordinates{(22/31,30/31) (55/72,1) (121/150,1)};
	
\end{axis}
\end{tikzpicture}
\caption{The intuition for the bound in Eq.~\eqref{eq:prop:loss_formula}.
The maximal difference between $v(c)$ and $\tilde{v}(c)$ is smaller than the maximal difference between the dashed and the solid lines, which is attained at $c_0$.}\label{Fig:BoundOnBeingFixed}
\end{center}
\end{figure}

In \thref{prop:loss}, we construct a bound by comparing a lower bound on $v(c)$ with an upper bound on $\tilde{v}(c)$.
This upper bound is constructed by using two feasible and simple static  strategies: the minimax pure action of game $A$ and an optimal mixed action of game $A$, the one with the lowest switching cost.
Any convex combination of these two static  strategies is also a static  strategy and can also be considered to improve the bound.

In particular, assume $y^*$ is, for Player~2 in game $A$, the optimal mixed action with the lowest switching costs ($y^*=\argmin\limits_{y\in\mathcal{A}} y^T S y$) and $\bar{y}$ is the pure minimax action of Player~2 (so $\bar{y}^TS\bar{y}=0$).
For every $\alpha\in [0,1]$, $y_\alpha =\alpha y^* + (1-\alpha ) \bar{y}$ is also a feasible static  strategy for Player~2 and its payoff bounds $\tilde{v}(c)$ from above.
\begin{theorem}\thlabel{thm:Bound_using_alpha_mixture}
Let $\bar{y}$ be a pure minimax action in game $A$, and denote by $\hat{s}={y^*}^TS\bar{y}+\bar{y}^TSy^*$.\footnote{Different bounds can be attained by using different pure minimax actions (if there are several). In that case, all of them can be constructed using the method presented, and the lowest chosen.}
If $s>\hat{s}$, then for every $c\in [c_1,c_2]\cap (\ubar{c},\bar{c})$,
the bound from \thref{prop:loss} can be improved by using 
\begin{equation}\label{eq:Bound_using_alpha_mixture}
\Delta (c)=(\bar{v}-v-\ubar{c}s)\tfrac{\bar{c}-c}{\bar{c}-\ubar{c}}-\tfrac{(\bar{v}-v-c\hat{s})^2}{4c(s-\hat{s})},
\end{equation}
where $c_1=\tfrac{s}{2s-\hat{s}}c_0$ and $c_2=\tfrac{s}{\hat{s}}c_0$.
\end{theorem}

\begin{proof}
For every $x$, the payoff of the strategy pair $(x,y_\alpha)$ is
\begin{eqnarray}
g(c)(x,y_\alpha)&=&xAy_\alpha + c y_\alpha^T S y_\alpha \nonumber \\
&=& \alpha xAy^*+(1-\alpha) xA\bar{y} +c\alpha^2 {y^*}^T S y^* + c\alpha(1-\alpha)\left[{y^*}^TS\bar{y}+\bar{y}^TSy^*\right] . \nonumber
\end{eqnarray}
Since $y^*$ is optimal in $A$, it guarantees the value and $xAy^*\leq v$.
Similarly, $\bar{y}$ guarantees the minimax value so $xA\bar{y}\leq \bar{v}$.
As previously, we let $s={y^*}^T S y^*$ and to simplify the expressions also $\hat{s}={y^*}^TS\bar{y}+\bar{y}^TSy^*$ (note that if $S$ is symmetric, $\hat{s}=2{y^*}^TS\bar{y}$).
Therefore $g(c)(x,y_\alpha)$ can be bounded by the following quadratic function of $\alpha$:
\begin{equation}
g(c)(x,y_\alpha)\leq \alpha v+(1-\alpha)\bar{v}+\alpha^2 cs+ \alpha(1-\alpha) c \hat{s}=\alpha^2 c(s-\hat{s})+\alpha(c\hat{s}+v-\bar{v})+\bar{v}\coloneqq f(\alpha). \nonumber
\end{equation}
Since $\tilde{v}(c)\leq g(c)(x,y_\alpha)$, it follows that $\tilde{v}(c)\leq \min\limits_{\alpha\in [0,1]} f(\alpha)$.
Hence, by finding this minimum, we can obtain a tighter bound on $\tilde{v}(c)-v(c)$ relative to \thref{prop:loss}.

If $f$ is concave, the minimum is obtained on one of the extreme points (depending on $c$), which takes us back to the case of \thref{prop:loss}.
In fact, we depart from the case presented in \thref{prop:loss} only if $f$ is convex and the minimum is inside $(0,1)$, i.e. only if $s>\hat{s}$ and $0\leq \tfrac{\bar{v}-v-c\hat{s}}{2c(s-\hat{s})}\leq 1$.

For the rest of the proof we assume therefore that $s>\hat{s}$.
The condition $0\leq \tfrac{\bar{v}-v-c\hat{s}}{2c(s-\hat{s})}\leq 1$ can be rearranged as a condition on $c$.
The ``$0\leq$'' part becomes $c\leq c_2=\tfrac{\bar{v}-v}{\hat{s}}$ and the ``$\leq 1$'' becomes 
$c\geq c_1= \tfrac{\bar{v}-v}{2s-\hat{s}}$.

Hence, for $c\in [c_1,c_2]\cap (\ubar{c},\bar{c})$, the bound from \thref{prop:loss} can be improved by using 
\begin{equation}
f\left(\tfrac{\bar{v}-v-c\hat{s}}{2c(s-\hat{s})}\right)=\bar{v}-\tfrac{(c\hat{s}+v-\bar{v})^2}{4c(s-\hat{s})} \nonumber
\end{equation}
instead of $\min \{\bar{v},v+cs\}$.
Thus, the bound for this domain is
\begin{equation}
\Delta (c)=(\bar{v}-v-\ubar{c}s)\tfrac{\bar{c}-c}{\bar{c}-\ubar{c}}-\tfrac{(\bar{v}-v-c\hat{s})^2}{4c(s-\hat{s})}. \nonumber
\end{equation}
\hfill\end{proof}

An illustration of this bound can be found in Figure \ref{Fig:BoundOnBeingFixed_with_mixed}, applied to \thref{ex_tildevc_nonlinear}.
Interestingly, in this case the optimal static  strategy is actually a mixture of $y^*$ and $\bar{y}$, so the bound is tight and achieves the minimax value in static  strategies.
We demonstrate this bound through \thref{ex_tildevc_nonlinear} due to the strong variance in the switching-costs matrix.
Without this variance, this bound might not be applicable, as shown in \thref{cor:NO_Bound_using_alpha_mixture_fixedS} for equal switching costs ($s_{ij}=1$ for all $i\neq j$).

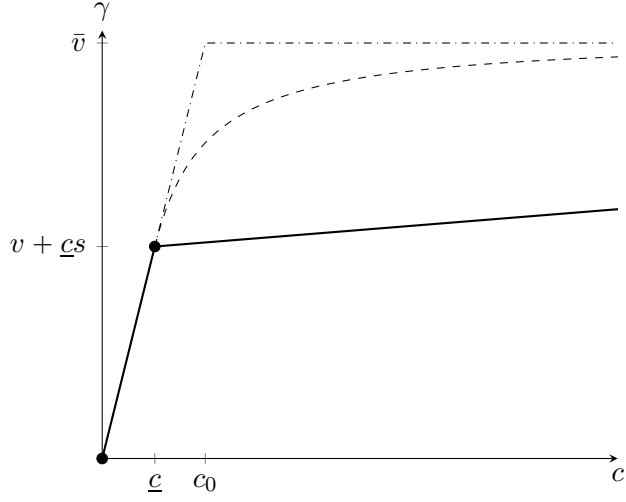
\begin{figure}
\begin{center}
\begin{tikzpicture}

\begin{axis}[axis x line=middle, axis y line=middle, xmin=0,  xmax=0.1, ymin=0, ymax=1.03, xtick={1/98, 1/50, 1/2}, xticklabels={$\ubar{c}$,$c_0$,$\bar{c}$}, ytick={0,25/49, 1}, yticklabels={$v$,$v+\ubar{c}s$,$\bar{v}$}, xlabel={$c$}, ylabel={$\gamma$},  x label style={at={(axis description cs:1,0)},anchor=north},    y label style={at={(axis description cs:0,1)},anchor=south}]
	\addplot[mark=*, only marks] coordinates {(0,0) (1/98,25/49) (1/2,1)}; 
	\addplot[thick] coordinates{(0,0) (1/98,25/49) (1/2,1) (1,1)};
	\addplot[dashdotted] coordinates{(0,0) (1/50,1) (1,1)};
	\addplot[dashed, domain=1/98:0.5, samples=400]  {1-((1-2*\x)^2/(192*\x))};
	
\end{axis}
\end{tikzpicture}
\caption{The bound of \thref{thm:Bound_using_alpha_mixture}, applied to Example \ref{ex_tildevc_nonlinear}.
The solid line is $v(c)$, the dotted line is the bound (in this example, it is also $\tilde{v}(c)$) and the dot-dashed line is the bound from  \thref{prop:loss}.
For clarity, the graph was zoomed on the region near $c_0$ where the difference is the largest. All three meet at $\bar{c}$.}\label{Fig:BoundOnBeingFixed_with_mixed}
\end{center}
\end{figure}

\subsection{Action-Independent Switching Costs}\label{s_constantsc} 

In this section, we consider the special case of switching costs which are independent of the actions being switched (uniform switching costs), i.e. $s_{ij}=1$ for all $i\neq j$.
This case naturally arises when the switching stem from the act of ``switching'' itself and do not depend on the actions being switched to or from.
For example, consider a person trying to multi-task.
Each time he switches from one task to another, it takes him a short time to regain focus on the task he is about to perform.
This time is generally fixed and does not depend on the nature of the task (whereas the time spent actually performing different tasks can differ significantly).
This case was explored in previous papers, such as \cite{lipman2000switching,lipman2009switching} and \cite{schoenmakers2008repeated}.
In fact, the model \cite{schoenmakers2008repeated} studied is exactly the one presented here, up to an additive constant (as they pay a positive bonus for not switching actions and nothing when actions are switched).
We refer the interested reader to \cite{schoenmakers2008repeated}, who provide additional results using this model.
For example, they analyze in depth the case where $A$ is an $m\times 2$ matrix and provide a formula to compute the optimal strategy when $A$ is a square matrix.

A significant benefit from the uniform switching-costs assumption is that finding the optimal static  strategy becomes simpler.
Indeed, note first that under this assumption, $g(c)(x,y)$ can be significantly simplified.
Let $\1$ be the ones matrix (for all $i,j:\,(\1)_{ij}=1$).
Then $S=\1-I$ and 
\begin{equation}\label{eq:ySy_for_constantS}
y^tSy=y^t(\1-I)y=y^t\1y - ||y||^2=1-||y||^2. 
\end{equation}
So $g(c)(x,y)=x^TAy+c(1-||y||^2)$ and $\tilde{v}(c)=c+\min\limits_{y \in \Delta(\{1,\ldots,n\})} \left\{\max\limits_{x \in \Delta(\{1,\ldots,m\})} \{x^T Ay\} - c ||y||^2 \right\}$.
Using this expression we can strengthen \thref{lem:propertiesoftildevc}, as it turns out that now $\tilde{v}(c)$ is piece-wise linear.
In a similar fashion to \thref{cor:optimalstationaryispiecewise}, this implies the existence of a finite set of strategies to consider when looking for the optimal static  strategy. 

\begin{proposition}\thlabel{lem:tildevc_pw_linear_constantS}
Whenever switching costs are uniform \emph{(}$s_{ij}=1$ for $i\neq j$ and $0$ otherwise\emph{)}, the minimax value $\tilde{v}(c)$ in static  strategies is a continuous, increasing, and concave piece-wise linear function.
\end{proposition}

\begin{proof}
It follows from \thref{lem:propertiesoftildevc} that $\tilde{v}(c)$ is continuous, increasing, and concave. 
Now we prove that it is also piece-wise linear.
For each $i\in \{1,\ldots,m\}$, let $I_i$ be the subset of $\Delta(\{1,\ldots,n\})$ such that $i$ is a best response for Player~1 in the one-shot game $A$ to all $y\in I_i$.
It is well known that $I_i$ is a compact, convex polyhedron in $\RR^n$.

Now, suppose Player~2 is restricted to playing only in set $I_i$.
The best response for Player~1 is $i$ and the payoff (based on the above calculation) is
$c+\min\limits_{y \in I_i} \left\{A_iy - c ||y||^2 \right\}$, where $A_i$ is the $i$th row of $A$.
The expression inside the minimization is a concave function in $y$ and the minimization domain is a polyhedron, so the minimum is attained at one of its vertices.

There are $m$ polyhedrons and each one has finitely many vertices, so the set of candidate optimal strategies is finite.
Using a similar argument as in \thref{lem:proportiesofvc}, it follows that $\tilde{v}(c)$ is piece-wise linear.\hfill\end{proof}

The reasoning behind this proof can be used to slightly strengthen \thref{lem:propertiesoftildevc} (even without the additional assumption on $S$).
If we restrict Player~2 to $I_i$ and write $g(c)(x,y)$ in the general case explicitly (including the constraint $\sum y_i=1$), we obtain a function which is concave in each of the variables $y_i$.
Thus, the $\argmin$ cannot be an interior point of $I_i$ and must lie on the boundary.
However, since the function is not concave, the $\argmin$ is not guaranteed to be  on a vertex, and could be on the boundary, where it moves as $c$ changes, as in \thref{ex_tildevc_nonlinear}.
Still, the fact that the $\argmin$ is on the boundary of one of the $I_i$s means that when playing an optimal static  strategy, Player~2 keeps Player~1 indifferent between at least two of her actions, and this can be used to reduce the dimension of the search.

As mentioned, a direct result of \thref{lem:tildevc_pw_linear_constantS} is that the optimal strategy comes from a finite set and is robust to the exact value of $c$.
Unlike \thref{cor:optimalstationaryispiecewise} however, this time it is also true for the best responses of Player~1 (since she best-responds purely).
This provides a theoretical explanation for the numerical results of  \cite{rass2014numerical} and \cite{GDO2020paper}, which indicate that slightly changing $c$ (or $\alpha$, in their model) has a small effect, if any, on optimal strategies.

\begin{corollary}\thlabel{cor:optimalfixedispiecewise}
When switching costs are uniform ($s_{ij}=1$ for $i\neq j$ and $0$ otherwise), there exists $K\in \NN$ and $0=c_0<c_1<c_2<\ldots<c_K<c_{K+1}=\infty$ such that for each segment $[c_i,c_{i+1}]$ there exists a Player~2 strategy to achieve  the minimax value $\tilde{v}$ in static strategies that is suited to all $c$s in the segment.
\end{corollary}
\noindent The proof is identical to the proof of \thref{cor:optimalstationaryispiecewise}.
Note that in general the number of segments differs between $v(c)$ and $\tilde{v}(c)$.

According to \thref{lem:tildevc_pw_linear_constantS} and \thref{cor:optimalfixedispiecewise}, there is a finite set of static  strategies from which an optimal static  strategy can be selected for any $c$.
This set can be large (exponential in $n$), and the above results can be used to search more efficiently for the optimal strategy within the set.
For example, by calculating $y^T S y$ for each $y$ in the set, the strategies in the set can be ordered from largest to smallest.
Since $y^T S y$ are the possible slopes of $c$ in $\tilde{v}(c)$ and since $\tilde{v}(c)$ is concave, if a particular $y$ is found to be optimal for a specific $c$, all $y$s with a higher slope can be excluded from the search domain for all larger $c$s.
A similar approach can be used to find the optimal stationary strategy as a consequence of \thref{lem:proportiesofvc}, keeping in mind that calculating the slope of $c$ for a given stationary strategy is slightly more complicated.

We can improve the bound from \thref{thm:uniformloss} by plugging in the explicit expression of $S$.

\begin{theorem}\thlabel{cor:constantS_bound}
When switching costs are uniform \emph{(}$s_{ij}=1$ for $i\neq j$ and $0$ otherwise\emph{)}, \emph{Player~2}'s loss from limiting himself to static  strategies, $\tilde{v}(c)-v(c)$, is bounded by
\begin{equation}\label{eq:cor:loss_formula}
\Delta(c) = \tfrac{1}{2}\left(1+c(1-\tfrac{1}{n})-v(c)\right).
\end{equation}
\end{theorem}
\begin{proof}
The matrix $S$ is symmetric, so $\Xi=0$.
Moreover, $y^{T} S y=1-||y||^2$ so $M=\max\limits_{y\in \Delta(\{1,...,m\})} (1-||y||^2)$.
Inside the simplex, the maximum is achieved at $y=(\tfrac{1}{n},\ldots,\tfrac{1}{n})$ and is equal to $1-\tfrac{1}{n}$.
Plugging these into \thref{thm:uniformloss}, the bound becomes
$\Delta(c) = \tfrac{1}{2}\left(1+c(1-\tfrac{1}{n})-v(c)\right)$.
\hfill\end{proof}

Interestingly, this bound cannot be improved by harnessing \thref{thm:Bound_using_alpha_mixture} to this case.
The reason is that the condition $s>\hat{s}$ can never hold in this kind of switching-costs matrix, and requires inequality of some kind in the costs (see, for example, \thref{ex_tildevc_nonlinear} and Figure \ref{Fig:BoundOnBeingFixed_with_mixed}).

\begin{obs}\thlabel{cor:NO_Bound_using_alpha_mixture_fixedS}
When switching costs are uniform ($s_{ij}=1$ for $i\neq j$ and $0$ otherwise), the condition $s>\hat{s}$ from \thref{thm:Bound_using_alpha_mixture} never holds. 
\end{obs}
\begin{proof}
In this case, we can compute $s$ and $\hat{s}$ directly.
Following Eq.~\eqref{eq:ySy_for_constantS}, $s=1-||y^*||^2$ and by definition, $\hat{s}=2(1-\bar{y}\cdot y^*)$.\footnote{Since $\bar{y}$ is pure, this dot product is simply the probability of choosing pure action $\bar{y}$ when following mixed $y^*$.}
The condition $s>\hat{s}$ from \thref{thm:Bound_using_alpha_mixture} becomes $1-||y^*||^2>2(1-\bar{y}\cdot y^*)$ or equivalently $(2\bar{y}-y^*)\cdot y^*>1$.

However, the maximum of the function $(2\bar{y}-x)\cdot x$ (over all $\RR^n$) is at $x=\bar{y}$, and is equal to $||\bar{y}||^2=1$.
Thus, $(2\bar{y}-y^*)\cdot y^*\leq 1$ and the desired inequality never holds.
\hfill\end{proof}

\section{Static Strategies in Stochastic Games}\label{s_gen}

So far, we have compared the values guaranteed by stationary and static  strategies in a particular class of games resulting from adding switching costs to normal-form repeated games.
In this section, we consider general stochastic games and study which game structure is needed for our results (or similar) to hold. 
Naturally, we cannot consider the most general stochastic games and must at least assume that the same set of actions is available to the players in each state; otherwise, static  strategies are not well defined.

Let $\Gamma$ be a stochastic game.
We denote by $S$ the finite set of states when $I$ and $J$ are finite action sets of Player~1 and Player~2 (resp.).
The transition function is denoted by $q:S\times I \times J \rightarrow \Delta(S)$ and the normalized payoff function for Player~1 (zero-sum) by $r: S\times I \times J \rightarrow [0,1]$.
We assume that the normalization is tight: $\min\limits_{s,i,j}r(s,i,j)=0$ and $\max\limits_{s,i,j}r(s,i,j)=1$.

As before, a stationary strategy is a strategy that depends at each $t$ on the state, but not on the rest of the history or on $t$ itself.
We denote by $\Sigma$ the set of stationary strategies for Player~1, and by $\Tau$ the set of stationary strategies for Player~2.
A static  strategy is a strategy that dictates the same mixed action in each stage, independent of $t$ or the history.
We denote by $\Sigma_{\f}$ the set of static  strategies for Player~1, and by $\Tau_{\f}$ the set of static  strategies for Player~2.

Without additional requirements, limiting Player~2 to static  strategies can be arbitrarily bad for him, as the following example shows:

\begin{example} An arbitrarily bad payoff in static  strategies relative to stationary strategies. \rm

Consider the following three-state stochastic game.
For simplicity, Player~1 has only one action and has no role in the game.
Arrows represent the deterministic transitions, so $s_R$ is an absorbing state:

\[s_L:
\begin{pmatrix}
1 \circlearrowleft, &
0 \rightarrow \\
\end{pmatrix}
\qquad
s_M:
\begin{pmatrix}
0 \leftarrow, & 0 \rightarrow
 \\
\end{pmatrix}
\qquad
s_R:
\begin{pmatrix}
1 \circlearrowleft, &
1 \circlearrowleft \\
\end{pmatrix}
\]

It is clear that the optimal strategy for Player~2 is the stationary strategy that plays $R$ at $S_L$ and $L$ at $S_M$, since it guarantees him the minimal payoff in this game.
By contrast, let us consider a static  strategy that plays $L$ with probability $x$ (and $R$ otherwise).
If $x<1$, then every trajectory eventually reaches state $s_R$ where the payoff is $1$, whereas if $x=1$, the game remains in $s_L$ and the payoff is $1$.
It follows that the difference between the two values is equal to the maximum difference in payoff.\hfill$\Diamond$
\end{example}

We therefore consider only stochastic games where only Player~2 controls the states and the states correspond to the previously played action.
Hence, $S=J$ and 
every transition satisfies
\begin{equation}
\forall (s,i,j)\in S\times I \times J, \  q(s,i,j)= \delta_{j}. \nonumber
\end{equation}
Informally, the state at stage $t+1$ is the action played by Player~$2$ at stage $t$. 
This class of games is denoted by $\mathcal{G}$.
Moreover, we denote by $\mathcal{G}_{S}$ the set of games in $\mathcal{G}$ such that there exist two matrices $A$ of size $|I| \times |J|$, and $S$ of size $|J| \times |J|$  such that
\begin{equation}
\forall (j',i,j)\in S\times I \times J, \ r(j',i,j)=A_{i,j}+S_{j',j}. \nonumber
\end{equation}

Thus, Sections \ref{s_model} and \ref{s_results} dealt only with games from $\mathcal{G}_{S}$ (with the additional constraints $s_{ij}\geq 0$ and $s_{ii}=0$).
In \thref{exmple_SGnotinG_S} we show that indeed $\mathcal{G}\neq\mathcal{G}_S$.
The following characterization of $\mathcal{G}_{S}$ can be used to see whether a game in $\mathcal{G}$ is indeed in $\mathcal{G}_S$ and how to construct the appropriate matrices $A$ and $S$.
The intuition behind this result is that there must be no interaction between the actions of the maximizing player and the states.

\begin{proposition}\thlabel{prop:G_neq_Gs}
Let $\Gamma$ be in $\mathcal{G}$, then $\Gamma$ is in $\mathcal{G}_{S}$ if and only if
\begin{equation}
\forall i,i' \in I, \ \forall j,j',j'' \in J, \  r(j',i,j)-r(j'',i,j)=r(j',i',j)-r(j'',i',j).
\end{equation}
\end{proposition}

\begin{proof}
Let us first check that if $\Gamma$ is in $\mathcal{G}_{S}$, then it satisfies the equation. Let $i \in I$ and $j,j',j'' \in J$, then
\begin{equation}
r(j',i,j)-r(j'',i,j)=(A_{ij}+S_{j'j})-(A_{ij}+S_{j''j})=S_{j'j}-S_{j''j}. \nonumber
\end{equation}
Hence, it is independent of $i$ and therefore for two actions $i$ and $i'$ of Player~1, the two differences are equal.

Let us now prove the converse. 
Let $r$ be a payoff function of a stochastic game in $\mathcal{G}$ that satisfies
\begin{equation}
\forall i,i' \in I, \ \forall j,j',j'' \in J, \  r(j',i,j)-r(j'',i,j)=r(j',i',j)-r(j'',i',j). \nonumber
\end{equation}
Fix $i_0\in I, j_0\in J$. Define for every $j\in J$ and for every $j'\in J$,
\begin{equation}
S_{jj'}=r(j',i_0,j)-r(j_0,i_0,j).\nonumber
\end{equation}
Notice that $S_{jj_0}=0$ for every $j\in J$.  Moreover for every $(j',i,j) \in S\times I\times J$, let 
\begin{equation}
B_{ijj'}=r(j',i,j)-S_{jj'}=r(j',i,j)-r(j',i_0,j)+r(j_0,i_0,j). \nonumber
\end{equation}

To conclude, we need to check that $B_{ijj'}$ does not depend on $j'$ and is therefore equal to some $A_{ij}$.
To do so, letting $j'_0$ and $j'_1$, we have
\begin{align}
& B_{ijj'_0}-B_{ijj'_1} \nonumber\\
&=[r(j'_0,i,j)-r(j'_0,i_0,j)+r(j_0,i_0,j)]-[r(j'_1,i,j)-r(j'_1,i_0,j)+r(j_0,i_0,j)] \nonumber\\
& =[r(j'_0,i,j)-r(j'_1,i,j)]-[r(j'_0,i_0,j)-r(j'_1,i_0,j)]+[r(j_0,i_0,j)-r(j_0,i_0,j)] \nonumber\\
& =0. \nonumber
\end{align}
Hence the result.
\hfill\end{proof}

Given a game $\Gamma$ in $\mathcal{G}$, a pair of stationary strategies $(\sigma,\tau)\in \Sigma \times \Tau$, and an initial state $s\in S$, we define
\begin{equation}
\gamma(s,\sigma,\tau)=\E_{s,\sigma,\tau}\left( \liminf_{T\rightarrow +\infty}\frac{1}{T} \sum_{t=1}^{T} r(s_t,i_t,j_t) \right),\nonumber
\end{equation}
and we can define $\forall s\in S$,
\begin{equation}
v_{\Gamma}(s)= \inf\limits_{\tau \in \Tau} \sup\limits_{\sigma\in \Sigma} \gamma(s,\sigma,\tau).\nonumber
\end{equation}
which is the minimax value of Player~2 when using stationary strategies.
Note that the value is independent of the original state due to the structure of the transitions, so it will be dropped hereafter.

We also consider
\begin{equation}
\tilde{v}_{\Gamma}= \inf\limits_{\tau_f \in \Tau_f} \sup\limits_{\sigma\in \Sigma} \gamma(\sigma_{\f},\tau), \nonumber
\end{equation}
which is the minimax value of Player~2 when Player~2 is restricted to static  strategies.
Note that in this case, even when Player~2 uses a static  strategy, the best response of Player~1 might be stationary and not static.
\begin{example}A stochastic game where the best response of Player~1 to a  static  strategy of Player~2 is not static  \thlabel{exmple_SGnotinG_S}\rm

Consider the following two-state stochastic game
\[s_L:
\begin{pmatrix}
 1 \circlearrowleft & 0 \rightarrow  \\
 1 \circlearrowleft & 1  \rightarrow\\
\end{pmatrix}
\qquad
s_R:
\begin{pmatrix}
1 \leftarrow & 1\circlearrowleft \\
0 \leftarrow & 1 \circlearrowleft \\
\end{pmatrix}
\]
and suppose that Player~2 plays the static  non-pure action $[x(L),1-x (L)]$ for $x\in (0,1)$.
The expected payoff of $T$ in state $s_L$ is $x<1$ and the expected payoff of $B$ in state $s_R$ is $1-x<1$.
Hence, each of these actions is dominated by the other action in the state and the best response for Player~1 is to play $B$ in state $s_L$ and $T$ in state $s_R$, with an overall payoff of $1$.
Any static  strategy by Player~1 would result in a payoff strictly smaller than $1$ (which is worse for him, as the maximizer).

In the model presented in Section \ref{s_model}, the best response to a static  strategy is static.
Since that is not true here, this example shows that $\mathcal{G}\neq\mathcal{G}_{S}$ and that the payoff function of this game cannot be written as two additive payoffs, one corresponding to a normal-form game and another to the switching costs (see also \thref{prop:G_neq_Gs}).
\hfill $\Diamond$
\end{example}

Our main result generalizes  \thref{thm:uniformloss} for the games in $\mathcal{G}$, rather than just in  $\mathcal{G}_S$.

\begin{theorem}\thlabel{prop:genquart}
Fix $\Gamma\in\mathcal{G}$. 
The minimax value $\tilde{v}_{\Gamma}$ \emph{Player~2} can guarantee in static  strategies is at least a quarter of the value $v_{\Gamma}$ he can guarantee in stationary ones: 
\begin{equation}
\tilde{v}_{\Gamma} -v_{\Gamma} \leq \Delta_\Gamma= \frac{3}{4}(1-v_{\Gamma}).
\end{equation}
This equation is equivalent to the following
\begin{equation}
\tfrac{1-\tilde{v}_{\Gamma}}{1-v_{\Gamma}} \geq \frac{1}{4}.
\end{equation}
\end{theorem}

\begin{proof}
Let $v_{\Gamma}$ be the value of the game (in stationary strategies).
There exists a state $j^*\in S$ and a mixed action $y^*\in \Delta (J)$ such that for all $i \in I$ we have $r(j^*,i,y^*)\leq v_{\Gamma}$ (otherwise, the payoff of Player~1 in each state can exceed $v_{\Gamma}$, as can the total value).
Consider the static  strategy $\tau_f$ of Player~2, who with probability $0.5$ chooses action $j^*$ and otherwise chooses an action according to mixed action $y^*$.
The payoff of this strategy against the stationary strategy $\sigma(s)$ of Player~1 is
\begin{align}
\tfrac{1}{4}r(j^*,\sigma(j^*),y^*) & +\tfrac{1}{4}r(j^*,\sigma(j^*),j^*)
+\tfrac{1}{4}\sum\limits_{s\in S}y^*(s)r(s,\sigma(s),j^*)+\tfrac{1}{4}\sum\limits_{s\in S}y^*(s)r(s,\sigma(s),y^*) \nonumber\\ 
&\leq \tfrac{1}{4} v_{\Gamma}+\tfrac{3}{4},\nonumber
\end{align}
where we bound the payoff by the maximal payoff $1$.
It follows that
\begin{equation}
\tilde{v}_{\Gamma} \leq \tfrac{1}{4} v_{\Gamma}+ \tfrac{3}{4}.\nonumber
\end{equation}
Therefore, we obtain by subtracting $v_{\Gamma}$ on both sides that
\begin{equation}
\tilde{v}_{\Gamma} -v_{\Gamma} \leq \frac{3}{4}(1-v_{\Gamma}),\nonumber
\end{equation}
or by subtracting $1$ on both sides and then dividing by $v_{\Gamma}-1$ that
\begin{equation}
\tfrac{1-\tilde{v}_{\Gamma}}{1-v_{\Gamma}} \geq \frac{1}{4}.\nonumber
\end{equation}
The gain obtained by a static  strategy (compared to the maximal payoff $1$) is at least $\tfrac{1}{4}$ of the gain obtained by a stationary strategy.

To finalize the proof, we use an example to show that this bound is indeed tight. Suppose $J=\{0,\ldots,4\}$, $I=\{1\}$ and the payoff function is $1$, except for  $r(i,1,i+1)=0$ and $r(4,0,1)=0$.
The optimal stationary strategy is to choose in state $i$  action $i+1$ (modulo $5$) with $v_{\Gamma}=0$.
Let $\sigma_f$ be a static  strategy that chooses action $i$ with probability $x_i$.
Then the payoff is $1-x_0x_1-x_1x_2-\ldots-x_4x_0$ and it is straightforward to verify that the minimum of this function in domain $x_i\geq 0$ and $\sum\limits_{i=0}^4 x_i=1$ is $\tfrac{3}{4}$.
\hfill\end{proof}

The above result gives a tight bound on the static  value compared to the stationary value, as well as offering a strategy that attains it in some cases.
For specific games, the bound can be tighter and the optimal static  strategy can be different.
A particular set of games where the bound is different are games where $\vert J \vert \leq 4$.
In this case, the difference between the values is at most $\Delta_\Gamma=\left(1-\tfrac{1}{\vert J \vert}\right)(1-v_\Gamma)$ provided that Player~2 can guarantee to pay at most $v_\Gamma$ in each state as a one-shot game.
This is why the example in the proof of \thref{prop:genquart} has $5$ actions.

\section{Conclusions}\label{s_conclusions}
In this paper, we analyze zero-sum games with switching costs and compare the value under optimal strategies (that are stationary) to an approximate solution with time-independent strategies (``static  strategies'').
Our main findings are summarized in a series of theorems which provide bounds on the difference between the two types of strategies for different scenarios and under various assumptions.
Our bounds can be used to assess whether the additional computational complexity is worth the extra payoff gain or, symmetrically, assess the price of playing sub-optimal but simple strategies when stationary ones are forbidden.

In addition, we study and characterize the value in stationary strategies as a function of the ratio between switching costs and stage payoff ($c$).
Our key finding is that this function, $v(c)$, is piece-wise linear.
This implies that the solutions are robust in the sense that an optimal stationary strategy for a particular $c$ is also optimal for all the $c$s in its segment.
Hence, knowing the exact value of $c$ is not essential to play optimally and the strategy should not be changed in response to small changes in $c$.

We generalize our work on repeated games to compare static  and stationary strategies for general stochastic games.
This generalization allows us to consider a wider range of games with more elaborate switching-cost structures.
Still, we are able to propose a bound on the price of being history-dependent.
This generalization is unique to our paper, and considers the widest possible set of repeated zero-sum games to which the comparison between history-dependent and stationary strategies is applicable.

Our analysis can serve as a basis for future work regarding non-zero-sum games since, when only one player's payoffs are considered, $v(c)$ is his minimax value.
We therefore portray the behavior of the individually rational payoffs of the players, which serves as a lower bound for the possible equilibrium payoff via a version of a Folk Theorem suitable for this model.
A particularly interesting model is that where switching costs paid to a ``bank'' are added to a zero-sum stage game.
The addition of switching costs transforms the game into a non-zero-sum game and facilitates cooperation in an otherwise zero-sum game.

Other possible extensions are games where the switching costs depend on the outcome of the previous round and not only on the action played.
For example, \cite{wang2014social} report that when playing rock-paper-scissors, players tend to repeat an action that led to a victory.
This can be modeled theoretically as a switching cost added to all the other actions. 
Although the model is significantly different, some of the techniques and observations of this paper might be useful in its analysis.
This and the above non-zero-sum models are left for future research.

\bibliographystyle{plainnat}
\bibliography{SC_bib}

\newpage
\appendix

\section{Average Cost Optimality Equation (ACOE)}\label{app_acoe}
In this Appendix we show how the value of the game and the optimal stationary strategies can be computed using the Average Cost Optimality Equation.
We achieve this by solving Example \ref{ex_123} in detail.
More about this method and proofs can be found in \cite{arapostathis1993discrete}, the references within, and similar papers.

In general terms, the algorithm consists of the following steps:
\begin{enumerate}
\item ``Guess'' a stationary strategy for Player~2.
Based on the ideas of \thref{lem:proportiesofvc}, this stationary strategy should in each state play a mixed action that keeps Player~1 indifferent among several of his actions and makes him prefer them over the rest in the one-shot game $A$.
This well-defines a finite set of actions to consider and ensures that the algorithm eventually terminates.

\item Player~1 does not control transitions, so once the strategy of Player~2 is determined, he can best respond by myopically best responding in each state.
This can be achieved even by pure actions.
Fix such a best response.

\item Let $q_{ij}$ be the probability of Player~2 playing action $j$ in state $s_i$, let $r_i$ be the value of one-shot game $A$ when the players use their mixed actions chosen for state $s_i$, let $v_i$ be the continuation payoff of the repeated game starting from state $s_i$ and let $\gamma$ be the value of the game (Eq.\eqref{eq:gamma_liminf}).
The ACOE connects these variables in the following manner:
\begin{equation}
\gamma + v_i=r_i +c\sum\limits_j q_{ij}s_{ij} + \sum\limits_j q_{ij} v_j,
\end{equation} 
where $s_{ij}$ is an element in the matrix $S$.

\item The above equation actually represents $n$ linear equations (one for each state) with $n+1$ variables: $\gamma$ and the $v_i$s.
W.l.o.g. set $v_1=0$ and solve the remaining equations to obtain the value of the game $\gamma$ and the continuation payoffs $v_i$.

\item Find the real mixed actions of Player~1.
To do so, treat each state as a one-shot game by writing the switching costs and the continuation payoffs in the payoff matrix.
Check that Player~1 has a best response to Player~2's action in this game that keeps Player~2 indifferent among the actions he plays and makes him prefer them over all other actions.

\item If it is possible to find such a strategy for Player~1, this pair of strategies is optimal and $\gamma$ is the value of the game.
Otherwise, restart the algorithm with another guess.
\end{enumerate}

A convenient way to see that the algorithm is finite is to assume the initial guess is a stationary strategy that mixes all the actions of Player~2 (in each state) and if it fails -- use another that does not use one pure action, then two pure actions and so on.
Eventually, the process terminates as there are only finitely many possible pure actions to stop considering.

We now apply this algorithm to the zero-sum game $A=\left(\begin{smallmatrix} 1 & 0 & 0\\ 0 & 2 & 0 \\ 0 & 0 & 3\end{smallmatrix}\right)$ and the switching-costs matrix $S=\left(\begin{smallmatrix} 0 & 1 & 1\\ 1 & 0 & 1 \\ 1 & 1 & 0 \end{smallmatrix}\right)$.
Without switching costs ($c=0$), the optimal strategy for Player~2 is $\tfrac{1}{11}(6,3,2)$ and the value is $v=\tfrac{6}{11}$.
For clarity, we denote the actions of Player~2 as $L$ (left), $M$ (middle) and $R$ (right), and the state after action $i\in\{L,M,R\}$ was played by $s_i$.

Suppose that Player~2 plays $\tfrac{1}{11}(6,3,2)$ in all states.
This makes Player~1 indifferent among all his actions, and we can assume that the best response is simply playing the first row.
This results in $r_i=\tfrac{6}{11}$ and $(q_{iL},q_{iM},q_{iR})=\tfrac{1}{11}(6,3,2)$ for all $i$.
The three ACOE equations are
\begin{eqnarray}
\gamma + v_L&=&\tfrac{6}{11} +c\tfrac{5}{11} + \tfrac{6}{11}v_L+\tfrac{3}{11}v_M+\tfrac{2}{11}v_R, \nonumber \\
\gamma + v_M&=&\tfrac{6}{11} +c\tfrac{8}{11} + \tfrac{6}{11}v_L+\tfrac{3}{11}v_M+\tfrac{2}{11}v_R ,\nonumber \\
\gamma + v_R&=&\tfrac{6}{11} +c\tfrac{9}{11} + \tfrac{6}{11}v_L+\tfrac{3}{11}v_M+\tfrac{2}{11}v_R . \nonumber
\end{eqnarray}
Setting $v_L=0$ and solving these equations leads to $v_M=\tfrac{3c}{11},v_R=\tfrac{4c}{11}$ and $\gamma=\tfrac{6}{11}+\tfrac{72}{121}c$.
Our final step is to add these continuation payoffs to the payoff matrix of each state and find optimal strategies for Player~1.
The resulting three one-shot games are:
\[s_L:\left(\begin{matrix} 1 & \tfrac{14c}{11} & \tfrac{15c}{11}\\ 0 & 2+\tfrac{14c}{11} & \tfrac{15c}{11} \\ 0 & \tfrac{14c}{11} & 3+\tfrac{15c}{11} \end{matrix}\right) \qquad 
s_M:\left(\begin{matrix} 1+c & \tfrac{3c}{11} & \tfrac{15c}{11}\\ c & 2+\tfrac{3c}{11} & \tfrac{15c}{11} \\ c & \tfrac{3c}{11} & 3+\tfrac{15c}{11} \end{matrix}\right) \qquad 
s_R:\left(\begin{matrix} 1+c & \tfrac{14c}{11} & \tfrac{4c}{11}\\ c & 2+\tfrac{14c}{11} & \tfrac{4c}{11} \\ c & \tfrac{14c}{11} & 3+\tfrac{4c}{11}  \end{matrix}\right) \]
In game $s_L$, solving the standard equation of indifference leads to a mixed strategy in which Player~1 plays his actions (from top to bottom) with probabilities $\tfrac{1}{121}(72c+66,33-41c,22-31c)$ (and similar solutions for the other states).
This mixed action is indeed a distribution for $c\leq\tfrac{22}{31}$ and for these values of $c$ we have found the value $v(c)=\tfrac{6}{11}+\tfrac{72c}{121}$ and the optimal strategies.

For $c>\tfrac{22}{31}$, action $R$ in state $s_L$ becomes dominated and Player~1 cannot keep Player~2 indifferent between it and the other actions.
We therefore assume Player~2 never plays action $R$ in state $s_L$.
As a result, Player~1 will never play his bottom action in $s_L$ and the new guess for a mixed action for Player~2 in state $s_L$ is $\tfrac{1}{3}(2,1,0)$.
In the rest of the states, action $R$ is not dominated so Player~2 continues to play mixed action $\tfrac{1}{11}(6,3,2)$ in them.

Repeating the same process with this strategy (which is not static  this time) leads to the next critical value of $c=\tfrac{121}{156}$.
Above it,  action $M$ in state $s_L$ (including, again, continuation payoffs and switching costs) is dominated by $L$ and so $s_L$ becomes absorbing.


\end{document}